\documentclass[a4paper,11pt]{amsart}

\usepackage{amsfonts}
\usepackage{amssymb}
\usepackage{mathrsfs}
\usepackage{graphics}
\usepackage{tikz}
\usetikzlibrary{matrix,arrows}

\newcommand{\C}{{\mathbb C}}
\newcommand{\R}{{\mathbb R}}
\newcommand{\Q}{{\mathbb Q}}
\newcommand{\Z}{{\mathbb Z}}
\newcommand{\N}{{\mathbb N}}
\newcommand{\D}{{\mathbb D}}
\newcommand{\Pp}{{\mathbb P}}
\newcommand{\CP}{{\mathbb C}{\mathbb P}}
\newcommand{\Hup}{\mathbb{H}}

\newcommand{\fol}{{\mathcal F}}
\newcommand{\calp}{{\mathcal P}}

\newcommand{\cald}{{\mathcal{D}}}
\newcommand{\psl}{{\rm PSL}\, (2,{\mathbb C})}
\newcommand{\pslR}{{\rm PSL}\, (2,{\mathbb R})}
\newcommand{\pslZ}{{\rm PSL}\, (2,{\mathbb Z})}

\newcommand{\tilf}{{\widetilde{\mathcal F}}}

\usepackage[normalem]{ulem}
\usepackage{color}

\def\picill#1by#2(#3)#4
{\vbox to #2
	{\hrule width #1 height 0pt depth 0pt
		\vfill\special{illustration #3 scaled #4}}}

\usepackage{hyperref}

\newtheorem{theorem}{Theorem}[section]
\newtheorem{prop}[theorem]{Proposition}

\newtheorem{lemma}[theorem]{Lemma}

\theoremstyle{definition}
\newtheorem{defi}[theorem]{Definition}
\newtheorem{ex}{Example}

\theoremstyle{remark}
\newtheorem{remark}[theorem]{Remark}

\begin{document}

\title[Foliated currents]{Examples of harmonic foliated currents and singular Levi-flats on the projective plane}

\author[M. Alkateeb]{Mohamad Alkateeb}
\address{Institut de Math\'ematiques de Toulouse ; UMR 5219, Universit\'e de Toulouse, 118 Route de Narbonne, F-31062 Toulouse, France.}
\email{\href{mailto:mohamad.alkateeb@math.univ-toulouse.fr}{mohamad.alkateeb@math.univ-toulouse.fr}}

\author[J. Rebelo]{Julio Rebelo}
\address{Institut de Math\'ematiques de Toulouse ; UMR 5219, Universit\'e de Toulouse, 118 Route de Narbonne, F-31062 Toulouse, France.}
\email{\href{mailto:rebelo@math.univ-toulouse.fr}{rebelo@math.univ-toulouse.fr}}


\subjclass[2010]{Primary 37F75, 37A; Secondary 32C07.}
\keywords{}

\begin{abstract}
	We provide examples of foliations on the complex projective plane $\CP^2$ carrying positive foliated harmonic currents
	whose supports coincide with singular Levi-flats which, in turn, can be chosen real-analytic (but non-algebraic) or
	merely continuous
	with fractal transverse nature.
	Furthermore, non-trivial examples as above can already be found among foliations of degree~$2$ and~$3$. In addition,
	the space of positive foliated harmonic currents for these foliations is fully characterised and it contains
	a unique harmonic (non-closed) current supported on the Levi-flat in question.
	Finally, we also provide examples of foliations carrying diffuse positive
	foliated closed currents related to a theorem due to Brunella as well as a general criterion for the existence
	of singular real analytic Levi-flats for Riccati foliations. 
\end{abstract}

\maketitle

\section{Introduction}

The purpose of this paper is to provide examples of foliations on the complex projective plane $\CP^2$ carrying -
dynamically interesting - positive foliated harmonic currents that can be detailed to a significant extent. These
will include currents whose supports are either singular real-analytic Levi-flats or 
continuous Levi-flats with fractal nature and Hausdorff dimension greater than~$3$.
Throughout this paper, a current $T$ is said to be {\it harmonic}
if it verifies $i\partial \overline{\partial} T =0$. In particular, {\it closed} currents are automatically harmonic.
However, in a suitable sense, there are few foliations on $\CP^2$
admitting non-trivial positive foliated closed currents whereas the existence of 
positive {\it foliated harmonic currents}\, is a rather general phenomenon.
For example, according to Fornaess and Sibony \cite{fornaesssibony-2}, every
foliation on $\CP^2$ having only hyperbolic singularities and leaving no algebraic curve invariant admits
a unique positive foliated harmonic current (and no closed ones). In particular, the minimal foliations
constructed in \cite{Loray} admit a unique
positive foliated harmonic current and its support
coincides with all of $\CP^2$.

In a kind of opposite direction, the simplest examples of foliations on $\CP^2$
carrying positive foliated {\it closed}\, currents are provided by foliations leaving invariant some algebraic curve
$C \subset \CP^2$: the integration current
over the (possibly singular) curve $C$ is positive, foliated, and closed.
In general, the study of (positive foliated) harmonic/closed currents is a very active area relying on a
variety of points of view as confirmed by the large literature including \cite{Garnett}, \cite{Candel}, \cite{BoNessim},
\cite{fornaessSibony}, \cite{fornaesssibony-2}, \cite{deroinkleptsyn}, \cite{garrandes}, \cite{sibonyetc},
\cite{nguyen}, \cite{dinhsibony}, \cite{sibony_uniqueness}
among several others (cf. the reference lists of the quoted papers).

However, despite a few remarkable theorems on their existence and uniqueness such as Fornaess-Sibony theorem
in \cite{fornaesssibony-2}, very little is known about the actual structure of these harmonic currents. For example,
dealing with {\it non-closed} harmonic currents, basically the only type of information so far available is the mentioned fact that
the harmonic currents associated with the foliations of \cite{Loray} have full support.
In particular, this raises the outstanding problem of deciding whether or not these currents
are {\it absolutely continuous}.

Moreover, the complex projective plane is arguably the most interesting surface to investigate the structure of harmonic currents due to their relations
with a few outstanding conjectures that are specific for $\CP^2$, see for example \cite{moreonesibony}.
Not surprisingly, constructing suitable foliations on $\CP^2$ involves special difficulties which, in turn, 
explains the relative paucity of examples and our focus on this case.

Concerning the study of (non-closed) harmonic currents, the main results of this paper are summarized by Theorem~A which
seems to answer a question formulated in \cite{lebel-perez} and pertains directly some problems mentioned
in \cite{moreonesibony} (with respect to \cite{moreonesibony} see also our Theorem~B). This theorem can also be viewed as a general
statement about Riccati foliations on $\CP^2$.

\vspace{0.2cm}

\noindent {\bf Theorem A}. {\sl Let $\fol$ be a Riccati foliation on the projective plane
	$\CP^2$ having only simple invariant lines
	$C_1, \ldots , C_k$. Assume that the local holonomy map of the foliation $\fol$ around each invariant line $C_i$ is an elliptic element
of $\psl$ and that the global holonomy group $\Gamma \subset \psl$ of the foliation $\fol$
is a Fuchsian (resp. quasifuchsian) group of first kind. Then, the following holds:
\begin{itemize}
	\item[(a)] There exists a closed set $\mathcal{L} \subset \CP^2$ of topological dimension equal to~$3$
	which is a minimal set for the foliation $\fol$. In other words, $\mathcal{L}$ is invariant by $\fol$ and every
	leaf of $\fol$ in $\mathcal{L}$ is dense in $\mathcal{L}$.
	
	\item[(b)] If $\Gamma$ is Fuchsian, then $\mathcal{L}$ is a real analytic non-algebraic
	set with $k$ singular points, all of
	them of orbifold-type. If $\Gamma$ is quasifuchsian, then $\mathcal{L}$ is a ``singular topological manifold''
	with Hausdorff dimension strictly greater than~$3$ and whose singular set consists again of $k$ singular points
	of orbifold-type.
	
	\item[(c)] The set $\mathcal{L}$ carries a unique positive foliated harmonic current $T$. The current $T$
	is not closed and its support coincides with all of $\mathcal{L}$.
	
	\item[(d)] The space of all positive foliated harmonic currents on the surface $\CP^2$ is generated by
	$T$ and by the closed currents induced by integration over each of the invariant lines $C_1, \ldots ,C_k$.
	
	\item[(e)] The current $T$ has zero (geometric) self-intersection in the sense of \cite{fornaessSibony}.
\end{itemize}
}

In particular, the current $T$ in question
appears to be the first example of a (non closed, foliated)
positive harmonic current on $\CP^2$ supported on a set with empty interior or,
alternatively, on a set of null Lebesgue measure. In particular, it is interesting to see
that this support may be very regular (real-analytic) or merely continuous
with transverse fractal nature and Hausdorff dimension strictly greater than~$3$. We should also mention that
foliations as in Theorem~A associated with analytic Levi-flats already exist in degree~$2$ whereas
transversely fractal Levi-flats can be found among degree~$3$ foliations. The
reader is referred to Section~\ref{examplesetc} for explicit examples.

When the holonomy group $\Gamma$ is Fuchsian, the resulting Levi-flat $\mathcal{L}$ is necessarily
non-algebraic thanks to Lebl's theorem in \cite{lebel}: if it were algebraic, the
leaves of $\fol$ in $\mathcal{L}$ would be complex algebraic curves which is clearly not the case since they are
dense in $\mathcal{L}$ (item~(a)). It is also interesting to notice that the non-algebraic nature of the leaves of
$\fol$ in $\mathcal{L}$ is the only obstacle preventing $\mathcal{L}$ of being algebraic as it
follows from \cite{perez-mol-rosas}. Finally, note that
item~(b) in Theorem~A provides a criterion to ensure the existence of real-analytic Levi-flats.
The case where $\Gamma$ is Fuchsian provides {\it singular real-analytic}\, Levi-flats $\mathcal{L}$
with very simple singularities (orbifold-type): in some sense $\mathcal{L}$ is ``as little singular as possible''
which is worth singling out since it might shed some light on the well-known problem about the existence of smooth
Levi-flats on $\CP^2$.

At this point a comment of {\it closed foliated currents} is in order. Although they
rarely exist, the problem of describing those
situations where they are present has attracted significant interest due, among other reasons, to its connection
with Kobayashi hyperbolic spaces, see \cite{Demailly}, \cite{mcquillan}. With respect to the former, our ideas
also enable us to provide examples answering a question by N. Sibony which was motivated by Brunella's result in
\cite{brunella-L'EnsMath}.

First, recall that a (positive foliated)
current is said to be {\it algebraic}\, if it coincides with the integration current over an algebraic curve invariant by the foliation $\fol$.
Otherwise, $T$ is said to be {\it diffuse}. Recall also that a singular point
$p$ of a foliation $\fol$ is called {\it simple}\, if the foliation $\fol$ can locally be represented
by a holomorphic
vector field whose linear part at $p$ has at least one eigenvalue different from zero (cf. Section~\ref{basics}).
With this terminology, and building on the work of Mcquillan \cite{mcquillan}, Brunella showed
that any foliation $\fol$ on $\CP^2$
having only simple singular points and admitting a diffuse (positive foliated) closed current $T_A$ of Ahlfors type
must have degree at most~$4$, see \cite{brunella-L'EnsMath}. This has prompted N. Sibony
to wonder if some bound on the degree of the foliation $\fol$
may be obtained by dropping (or significantly weakening) the assumption on the corresponding singular points.
In this direction, in Section~\ref{examplesetc} of the paper, the following theorem will be proved:

\vspace{0.2cm}

\noindent {\bf Theorem B}. {\sl For every $n \in \N$, there exists a degree~$n$ foliation $\fol_n$ on $\CP^2$
satisfying all of the following conditions:
\begin{itemize}
	\item All but one of the singular points of the foliation $\fol_n$ are simple.
	
	\item The foliation $\fol_n$ carries uncountably many (independent) diffuse Ahlfors currents.
	Similarly, there are uncountably many
	holomorphic maps from $\C$ to $\CP^2$ that are tangent to the foliation $\fol_n$ and have Zariski-dense images.
\end{itemize}
}

Note, however, that not all (positive) diffuse foliated closed currents are of Ahlfors type and some examples
will also be indicated in Section~\ref{examplesetc}.

Let us close this Introduction with an outline of the paper. Given a discrete group action, the
standard {\it suspension construction} basically yields a
foliation defined on a certain fiber bundle. In this sense, suspensions never produce foliations on $\CP^2$. Yet,
our strategy will consist of building some very special suspensions and then systematically modify
the foliation/ambient surface to eventually make the whole structure
fit in $\CP^2$. Naturally, this will require us
to overcome a few difficulties such as the fact that our basic objects
(foliated currents, real analytic sets) {\it are not}\, birationally invariant in general.

In Section~\ref{basics}, we show how
to construct Riccati foliations on $\CP^1$-bundles with prescribed holonomy
(Proposition~\ref{Riccati_foliation-F1}), a result going all the way back
to Birkhoff \cite{Birkhoff-1} and rediscovered in \cite{LinsNeto}. Besides background and complementary material,
Section~\ref{basics} includes a reasonably short proof of Proposition~\ref{Riccati_foliation-F1}
which is similar to Lins-Neto's argument. However, besides making the paper self-contained, our
proof shows that certain additional choices are always possible and this freedom is sometimes useful in the course
of this work.

In Section~\ref{examplesetc}, we use Proposition~\ref{Riccati_foliation-F1} to provide examples of degree~$2$
foliations on $\CP^2$ satisfying the conditions of Theorem~A for Fuchsian groups and thus giving rise to
real-analytic Levi-flats. Here, it is worth
mentioning that the degree~$2$ examples
turn out to be the Halphen vector fields studied in \cite{adolfo-IHES}. Similarly, we show how to
obtain transversely fractal Levi-flats starting from degree~$3$. In the second part of this section, we prove
Theorem~B. Section~\ref{examplesetc} ends with some examples of diffuse positive foliated closed currents that {\it are not}\,
of Ahlfors type.
Section~\ref{provingTheoremA_Part_I} contains the proof of statements~(a) and~(b) in Theorem~A
except by a specific lemma (Lemma~\ref{addedlemmaanalytic}) whose proof is deferred to the last section
since similar techniques are employed to establish the remaining items in Theorem~A as well.

\noindent {\bf Terminology}. In the statements of our theorems and throughout the text,
${\rm PSL}\, (2,\C )$ is
identified with the automorphism group of $\CP^1$. Elements in ${\rm PSL}\, (2,\C )$ are
classified as hyperbolic, parabolic, elliptic and the identity. In other words,
the identity matrix is set aside so that whenever an element of ${\rm PSL}\, (2,\C )$ is said to be
elliptic or parabolic it is understood that {\it this element is different from the identity}.

\noindent {\bf Acknowledgments}. We would like to remember N. Sibony for his interest in our examples, his
untimely passage saddened us all. We also thank T.-C. Dinh and V.-A. Nguyen for their interest in this
work and for some nice
discussions about their results on intersection theory of currents.
Special thanks are due to R. Mol for explaining so much about (singular) Levi-flats to us.

Both authors are partially supported by CIMI through the project
``Complex dynamics of group actions, Halphen and Painlev\'e systems''.

\section{Riccati Foliations}\label{basics}

Let $R$ be a compact Riemann surface and $S$ a $\CP^1$-bundle over $R$ whose projection is denoted by $\calp : S \rightarrow R$.
A Riccati foliation is usually defined as a {\it singular foliation} $\fol$ on the surface $S$ that is {\it transverse} to the fibers of $S$.
This statement can be made accurate as follows. There is a finite set $\{p_1, \ldots ,p_k\} \subset R$ such that the following holds:
\begin{itemize}
  \item The fibers $C_i=\calp^{-1} (p_i)$ over the points $p_i$ are invariant by the foliation $\fol$, $i=1,\ldots ,k$.
  \item The leaves of the foliation $\fol$, away from the invariant fibers $C_i$, are transverse to the fibers of $\calp$.
  In particular, the singular points of the foliation $\fol$ are all contained in the union of the invariant fibers $C_i$.
\end{itemize}
There follows that the restriction of $\fol$ to the open surface $S \setminus \{C_1, \ldots ,C_k\}$, is regular and, in fact,
transverse to the fibers of $\calp$ sitting over points in the (open) Riemann surface $R \setminus \{p_1, \ldots ,p_k\}$. Thus,
the standard holonomy representation gives arise to a homomorphism $\rho : \pi_1 (R \setminus \{p_1, \ldots ,p_k\}) \rightarrow
{\rm Aut}\, (\CP^1) \simeq \psl$. The {\it global holonomy group}\,
$\Gamma = \rho (\pi_1 (R \setminus \{p_1, \ldots ,p_k\}))\subset \psl$ encodes all of the
transverse dynamics of the foliation $\fol$.

In this paper, however, we will focus on {\it classical Riccati foliations}\, in which case the base space
$R$ is again the complex projective line $\CP^1$ so that the surface $S$ is a $\CP^1$-bundle over $\CP^1$ and, hence,
a Hirzebruch surface $F_n$, $n\geq 0$. Recall that $F_0
=\CP^1 \times \CP^1$ and that, for $n \geq 1$, $F_n$ is fully characterised as a  $\CP^1$-bundle over
$\CP^1$ admitting a rational curve of self-intersection~$-n$ as a section, cf. \cite{barth}.

Recall that $F_n$ possesses a standard atlas consisting of four affine coordinates. More precisely,
consider two copies of $\C \times \CP^1$.
The first copy is endowed with a pair of affine coordinates, namely $(x,y)$ and $(\overline{x}, \overline{y})$,
satisfying with $x = \overline{x}$ and $y = 1/\overline{y}$.
The second copy if endowed with coordinates $(u,v)$ and $(\overline{u}, \overline{v})$ satisfying the analogous
relation. The surface $F_n$
can then be obtained by identifying the point $(x,y)$ of the first copy with the point $(u,v) = (1/x, x^ny)$ of the second
one. The section of self-intersection $-n$, with this
identification, is nothing but the rational curve defined by $\{ y=0\}$ ($\{ v=0\}$) and it will often
be referred to as the {\it null section}.

Consider now a Riccati foliation $\fol$ defined on the first Hirzebruch surface $F_1$. The condition of
transversality ensures that, in
the above mentioned affine coordinates $(x,y)$, the foliation $\fol$ is induced by a vector field having the form
\begin{equation}
F(x) \frac{\partial}{\partial x} + [c_0(x) + c_1(x) y + c_2(x) y^2] \frac{\partial}{\partial y},
\label{simplenormalform_Riccati_Fn}
\end{equation}
where $F$, $c_0$, $c_1$, and $c_2$ are polynomials. In particular, if $c_0$ is a non-zero constant then the foliation
$\fol$ is transverse to the null section of the surface $F_1$ except maybe at the ``point at infinity''. In this regard,
it is always possible to choose coordinates where the resulting fiber at infinity is not invariant by the
foliation $\fol$ so that the invariant fibers of the foliation $\fol$ are in natural
correspondence with the zeros of the polynomial $F$. Moreover, an invariant fiber is said to be {\it simple}
if it corresponds to a simple zero of $F$, otherwise it is called a {\it multiple fiber}.

An alternative definition of {\it simple fibers} that is slightly more intrinsic since it does not depend on the particular
system of coordinates, depends on the notion of eigenvalues of a singular point. Recall that on a complex
surface every holomorphic foliation $\fol$ is locally given by the integral curves of a holomorphic vector
field $X$ having only isolated zeros. A vector field satisfying
this condition is said to be a {\it local representative} of the foliation $\fol$. Now, if $p$ is a singular point
of the foliation $\fol$, the {\it eigenvalues}\, of the foliation $\fol$ at the singular point $p$ are
defined as the eigenvalues of the linear part of a representative vector field $X$ at the point $p$.
Since two representative vector fields differ by multiplication by an invertible function, it
follows that the eigenvalues of a foliation at the singular point 
$p$ are well defined only up to a multiplicative constant, cf. \cite{IlyashenkoYakovenko} and \cite{HelenaandI}.
With this terminology it is immediate to check
that an invariant fiber of a Riccati foliation is simple {\it if and only if every singular
point lying in this fiber has a non-zero eigenvalue in the direction transverse to the invariant fiber itself}.
Also, this notion does not depend on the choice of the
singular point in the sense that if one singular point has a non-zero eigenvalue in a direction transverse
to the invariant fiber then any other singular point lying in the same fiber does too. In this respect,
Equation~(\ref{simplenormalform_Riccati_Fn}) shows that an invariant fiber of a Riccati equation
contains at least one singular point and at most $2$ singular points.

Next, let $\Gamma$ be a subgroup of ${\rm PSL}\, (2,\C)$ along with a chosen generating set $\{M_1, \ldots ,M_{k-1}\}$.
Choose also (pairwise distinct) points $\{p_1, \ldots ,p_k\}$ in $\CP^1$ and set $B=\CP^1 \setminus \{p_1, \ldots ,p_k\}$.
The following proposition plays a basic role in our paper. This proposition goes back to Brikhoff
\cite{Birkhoff-1} and \cite{Birkhoff-2} though an independent treatment was provided by A. Lins-Neto in
\cite{LinsNeto}. For the convenience of the reader,
we formulate an accurate statement (Proposition~\ref{Riccati_foliation-F1}) below and provide a self-contained
proof that parallels the one given in \cite{LinsNeto}. Besides making the paper more self-contained, the
proof given here also singles out a certain amount of flexibility in the construction that will be
helpful in the course of the paper.

\begin{prop}
\label{Riccati_foliation-F1}
{\rm (cf. \cite{Birkhoff-1} and \cite{LinsNeto})}
With the above notation, there exists a Riccati foliation $\fol$ on the surface $F_1$ satisfying the following conditions: 
\begin{itemize}
  \item The foliation $\fol$ leaves invariant exactly~$k$ fibers sitting, respectively, over the points $\{p_1, \ldots ,p_k\}$. All these fibers
  are simple.
  \item For each $i \in \{1, \ldots ,k-1\}$, the local holonomy map arising from a small simple loop around the point $p_i$ coincides with the automorphism
  of $\CP^1$ identified with the matrix $M_i$. 
  \item The local holonomy map arising from a small simple loop around the point $p_k$ coincides with the automorphism of
  $\CP^1$ arising from the matrix $M_k=(M_{k-1} \ldots M_1)^{-1}$.
\end{itemize} 
In particular, the global holonomy group of the foliation $\fol$ coincides with $\Gamma$.
\end{prop}

\begin{proof}
For each $i \in \{1, \ldots ,k\}$, let $\gamma_i \subset \CP^1$ be a small simple loop around the point $p_i \in \CP^1$.
The fundamental group $\pi_1(B)$ of $B=\CP^1 \setminus \{p_1, \ldots ,p_k\}$ is generated by $\gamma_1, \ldots,\gamma_k$
along with the relation
$\gamma_1 \ast \cdots \ast \gamma_k = {\rm id}$. We then define a representation $\rho$ from $\pi_1 (B)$ in ${\rm PSL}\, (2,\C)$
by letting $\rho(\gamma_i)=M_i$ for $i=1, \ldots ,k$. The homomorphism $\rho$ is well defined since $M_k = (M_{k-1} \ldots M_1)^{-1}$
so that $M_k M_{k-1}\ldots M_1 = {\rm id}$. Also, by construction, we have $\rho(\pi_1(B))=\Gamma$.

Next, we use the standard {\it suspension}\, construction to obtain a $\CP^1$-bundle $N$ over $B$
equipped with a foliation $\cald$ which is transverse to its fibers and whose global holonomy group is $\Gamma$.
In fact, the holonomy map associated to $\cald$ and arising
from a small loop $\sigma_i \subset B$ encircling the missing point $p_i$ is precisely the automorphism induced by $M_i$. 

The manifold $N$ is clearly open since the basis $B$ is so. To obtain a compact manifold and an (extended)
singular foliation, we will ``fill in'' each of the missing fibers in $N$.
Clearly, denoting by $\calp : N \rightarrow B$ the bundle projection,
it suffices to show how to ``fill in'' the fiber $\calp^{-1} (p_1)$ over the point $p_1$.
For this, we consider a small disc $D$
around $p_1$ whose boundary $\partial D$ is identified with the loop $\sigma_1$. By means of a local coordinate
$u$, $D$ can be thought of as a disc
around $0 \in \C$. First, we will construct a Riccati foliation $\fol_1$ on the product $D \times \CP^1$ having a single invariant
fiber, which sits over $0 \in D \subset \C$, and whose holonomy is given by the matrix $M_1$.
To construct the foliation $\fol_1$, we then consider coordinates $(u,v)$ on $D \times \CP^1$ where
$u$ is as above and $v$ is an affine coordinate on $\CP^1$. Then the foliation $\fol_1$ is given by the
integral curves of the following vector fields
\[ \begin{cases} 
u \frac{\partial}{\partial u}-(\frac{1}{2\pi i}) \frac{\partial}{\partial v}, & \mbox{if $M_1$ is parabolic,} \\
u \frac{\partial}{\partial u}+av \frac{\partial}{\partial v}, & \mbox{otherwise.} \\
\end{cases}
\]
It is clear that the foliation above admits a holomorphic extension as a Riccati foliation to $D \times \CP^1$.
Moreover, if $M_1$ is not parabolic, then the corresponding holonomy map fixes the point $v=0$ and a direct
inspection shows that the multiplier of this fixed
point is $e^{2\pi i a}$ ($a \in \C^{\ast}$ and $\Re \, (a)  \in [-1,1]$, where $\Re \, (a)$ stands for the real part
of $a$). In turn, when $M_1$ is parabolic, the holonomy map is given by $v \mapsto v-1$ so that
it has a single fixed point corresponding to $v = \infty$ whose multiplier equals~$1$. In particular, note that the choice
$a=\pm 1$ leads to trivial holonomy.

To complete the ``filling'' of the ``missing'' fibers, it suffices to
show that it is possible to glue together the foliations $\cald$ and $\fol_1$ over the punctured disc $D^{\ast}$.
Gluing these foliations together amounts to constructing a holomorphic diffeomorphism $h$ from $D^{\ast} \times \CP^1$
to $\calp^{-1} (D^{\ast})$ taking the foliation $\fol_1$ to the foliation $\cald$. Though not indispensable,
at this point it is convenient
to remind the reader that every holomorphic $\CP^1$-bundle is holomorphically locally trivial owing to
a theorem due to Fisher and Grauert \cite{fisher}. To begin, we use the common coordinate $u$
and fix a base point $q \in D^{\ast}$. The fiber $(q,\CP^1)$ in $D^{\ast} \times \CP^1$ is endowed
with the coordinate $v$. To identify $(q,\CP^1)$ with the fiber $\calp^{-1} (q)$ we need to choose a affine coordinate $w$ on $\calp^{-1} (q)$. The choice
of $w$ is made as follows:
\begin{itemize}
	\item If $a=\pm 1$, i.e., the holonomy map coincides with the identity. Then $w$ is an arbitrary affine coordinate in $\calp^{-1} (q)$
	and the identification is $v=w$.

    \item If the holonomy map is neither the identity nor parabolic. Then we choose $a$ and $w$ so that the origin of $w$ coincides
    with a fixed point of the holonomy map having multiplier equal to $e^{2\pi i a}$ and set $v=w$.
    
    \item If the holonomy map is parabolic. Then the origin of $w$ is the fixed point of the holonomy map and the identification is $w=1/v$.
\end{itemize}
The diffeomorphism $h$ is then defined on $(q,\CP^1)$ by declaring that $h$ takes points of $(q,\CP^1)$ to points
in $\calp^{-1} (q)$ in accordance with the identification of the corresponding fibers. It then remains to extend $h$ to
$D^{\ast} \times \CP^1$. This is done by following the standard lifting path method for the
leaves of the foliations $\fol_1$ and $\cald$ (the latter restricted to $\calp^{-1} (D^{\ast})$). The path 
liftings are clearly possible since the restriction of the bundle projection to any (regular) leaf
of either $\fol_1$ or $\cald$ is a covering map of the base (the foliations are Riccati).
Furthermore, the diffeomorphism $h$ is globally well defined since its extension over a loop around
$0 \in \C$ coincides with the initial definition on $(q,\CP^1)$. In fact, $h$ on $(q,\CP^1)$ is
determined by the identification provided by the coordinates $v$ and $w$ so that it conjugates
the corresponding holonomy maps.

Summarising what precedes, the gluing construction described above allows us to define a Riccati
foliation, still denoted by $\cald$, on a compact surface $N'$ that happens to be a $\CP^1$-bundle over $\CP^1$.
Moreover, the Riccati foliation $\cald$ possesses exactly $k$ invariant
fibers and satisfies the required holonomy conditions. Finally, being a $\CP^1$-bundle over $\CP^1$,
the surface $N'$ is a Hirzebruch surface $F_n$, $n \geq 0$. Thus, to complete the proof of the proposition,
it only remains to check that we can assume without loss of generality that $N'$ is actually $F_1$.

Assuming the resulting surface $F_n$ is different from $F_1$, we will show how the surface can be modified
to become $F_1$ while keeping all the conditions on the transformed foliation (still denoted by $\cald$).
This will be done by constructing a convenient birational map
between $F_1$ and $F_n$. This birational map is obtained by composing successive (birational) maps
from $F_n$ to $F_{n-1}$ which are defined as follows. Pick a base point $q \in F_n$ which is away from
the null section (i.e., the section of self-intersection $-n$)
and lies in a fiber $\calp^{-1} (q_0)$ transversely intersecting the foliation $\cald$. The surface $F_n$ is then
blown up at $q$ so that the new surface $\widetilde{F}_n$ possesses now two rational curves of self-intersection~$-1$.
Namely, the exceptional divisor $E$ and the transform $\widetilde{\calp^{-1} (q_0)}$
of the fiber $\calp^{-1} (q_0)$. The blown up surface also possesses a rational curve of self-intersection~$-n$
given by the transform of the null section.
Next collapse the $(-1)$-curve $\widetilde{\calp^{-1} (q_0)}$ to obtain a new surface. We claim the following:

\vspace{0.2cm}

\noindent {\it Claim}. The surface obtained by collapsing the transform of $\calp^{-1} (q_0)$ is $F_{n-1}$ and the transform
of $E$ is a fiber of the corresponding fibration. The foliation $\cald$ induces a new foliation $\cald_1$ on
$F_{n-1}$ having an additional
invariant fiber, namely the transform of $E$. This invariant fiber, however, has trivial holonomy.

\vspace{0.2cm}

\noindent {\it Proof of the Claim}. It is clear that by collapsing of $\widetilde{\calp^{-1} (q_0)}$, the
transform of $E$ fits inside the previous fibration of $F_n$ so that the resulting surface is again a
$\CP^1$-bundle over $\CP^1$ and, hence, a Hirzebruch surface.
However, since the transform of the null section on $\widetilde{F}_n$ intersects $\widetilde{\calp^{-1} (q_0)}$
at a single point and transversely, its
image on the resulting Hirzebruch yields a section of self-intersection $-n+1$. Thus, the surface obtained
at the end of the procedure is, indeed, $F_{n-1}$, see \cite{barth}.

As to the foliation induced by $\cald$ on $F_{n-1}$, note that $E \subset \widetilde{F}_n$ is invariant by
the blown up foliation $\widetilde{\cald}$ since it comes from blowing up a regular point of $\cald$. For the same
reason, $E$ contains exactly one singular point of $\cald$ so that the regular part of $E$ (leaf of $\widetilde{\cald}$)
is simply connected and, hence, has trivial holonomy.
The remainder of the claim follows immediately from these two remarks.\qed

By successively applying the previous construction, we obtain a Riccati foliation $\cald'$ on the surface $F_1$ which satisfies
the required conditions in terms of holonomy but which possesses finitely many additional invariant fibers, each of them
carrying trivial holonomy. To finish the proof, we just need to show how these invariant fibers with trivial holonomy
can be ``eliminated''. This means that the foliation $\cald'$ can be modified, without changing its local holonomy maps,
to yield a new foliation for which the previous invariant fiber becomes a regular transverse fiber. The proof of
this assertion, however, is rather simple. It is again a gluing procedure. Fix then one such
fiber $\calp^{-1} (q)$. By means of a local coordinate $u$ on the base, we identify $\calp^{-1} (q)$ to the product
$D \times \CP^1$. In turn, on $D \times \CP^1$, we consider the horizontal foliation $\mathcal{H}$. Again, the proof
consists of gluing together the foliations $\cald'$ and $\mathcal{H}$ over $\calp^{-1} (D^{\ast})$
and $D^{\ast} \times \CP^1$. The construction of the gluing diffeomorphism
$h : D^{\ast} \times \CP^1 \rightarrow \calp^{-1} (D^{\ast})$
is, however, identical to the previous construction. In fact, the local holonomy maps arising from $\cald'$
and from $\mathcal{H}$ around the fibers in question are both trivial so that the construction carried out above
still applies. This ends the proof of Proposition~(\ref{Riccati_foliation-F1}).
\end{proof}

\begin{remark}
	\label{summarizingpoints_proposition}
	Note that the invariant fibers for the Riccati foliation $\fol$ constructed in Proposition~\ref{Riccati_foliation-F1}
	are all simple since the singular points of the foliation $\fol$ always have
	a non-zero eigenvalue associated with the direction transverse to the invariant fiber.
\end{remark}

The paper being mainly devoted to constructing foliations on the complex projective plane $\CP^2$, we need to make
accurate what is understood by a {\it Riccati foliation on $\CP^2$}.
Since the blow up of the surface $\CP^2$ at an arbitrary point leads to a surface isomorphic to the surface $F_1$,
the definition becomes very straightforward.

\begin{defi}\label{Riccatifoliations_onCP2}
A foliation on the complex projective plane $\CP^2$ is said to be a Riccati foliation if it is
obtained from a Riccati foliation $\fol$ on the Hirzebruch surface $F_1$ by collapsing the $(-1)$-rational curve
in $F_1$. In other words, a foliation $\fol_{\Pp}$
on $\CP^2$ is a Riccati foliation if it possesses a singular point whose blow up leads to a Riccati foliation on $F_1$.
\end{defi}

In what follows, we will mostly be interested in Riccati foliations on $\CP^2$ or on $F_1$, whose global
holonomy groups are certain specific subgroups of
$\psl$ along with a particular choice of generating set. Compared to the statement of
Proposition~(\ref{Riccati_foliation-F1}), the choice of points $p_1, \ldots ,p_k$ corresponding to the invariant fibers
is, however, of little importance. Thus, fixed a
group $\Gamma \subset \psl$ along with a generating set $\{M_1, \ldots ,M_{k-1}\}$,
the phrase {\it the Riccati foliation arising from $\Gamma$ and $\{M_1, \ldots ,M_{k-1}\}$} will be used to refer
to the Riccati foliation $\fol$ on the surface $F_1$ obtained from
Proposition~\ref{Riccati_foliation-F1} for an arbitrary choice of points $p_1, \ldots ,p_k$.
Alternatively, the reader may fix a particular choice throughout the text, for example, setting
$p_1 = 0$, $p_2 =1$, $p_3 = \infty$, and then $p_i = i$ for $i=4, \ldots ,k$.
Similarly, we will refer to the Riccati foliation on $\CP^2$ (arising from $\Gamma$ and $\{M_1, \ldots ,M_{k-1}\}$)
meaning the foliation on $\CP^2$ obtained by blowing down the previous Riccati foliation on $F_1$.
In terms of notation, Riccati foliations on $F_1$ will typically be denoted by $\fol$ whereas $\fol_{\Pp}$
will stand for Riccati foliations on $\CP^2$.

Let us close this section with a comment about the {\it degree of foliations}\, on $\CP^2$. First,
note that every homogeneous polynomial vector field $X$ on $\C^3$ induces a foliation on
$\CP^2$ unless $X$ is multiple of the radial vector field $R = x_1 \partial /\partial x_1
+ x_2 \partial /\partial x_2 + x_3 \partial /\partial x_3$. In fact, $X$ being homogeneous,
the direction associated with the vector $X(p)$,
$p \in \C^3 \setminus \{ (0,0,0)\}$ does not change over the radial line of $\C^3$ passing through $p$.
If, moreover, $X$ is not a multiple
of the radial vector field, then for a generic point $p$ the vector $X(p)$ induces a well defined direction
on $T_{q = P (p)} \CP^2$ where $P: \C^3 \setminus \{ (0,0,0)\} \rightarrow \CP^2$ stands for the canonical projection.
In particular, $X$ defines a singular foliation on $\CP^2$. Conversely, every singular holomorphic
foliation on the surface $\CP^2$ can be obtained out of a homogeneous polynomial vector field
on $\C^3$ by means of the preceding construction, see for example \cite{IlyashenkoYakovenko}.
If, in addition, we require the zero set of
the homogeneous vector field $X$ to have codimension at least~$2$, then every pair of homogeneous vector fields
$X$ and $X'$ inducing the same foliation on $\CP^2$ must have the same degree.
Thus, the following definition makes sense:

\begin{defi}\label{degreeoffoliations_onCP2}
	The degree of a foliation $\fol$ on $\CP^2$ is the degree of a homogeneous polynomial vector field $X$
	on $\C^3$ having singular set of codimension at least~$2$ and inducing $\fol$ on $\CP^2$ by means of
	the canonical projection $\C^3 \setminus \{ (0,0,0)\} \rightarrow \CP^2$.
\end{defi}

In closing this section, let us mention that the degree of a foliation $\fol$ on $\CP^2$ can alternatively
be defined as the number of tangencies of $\fol$ with a generic projective line, cf. \cite{IlyashenkoYakovenko}
and \cite{HelenaandI}. Finally, in the special case of Riccati foliations the following lemma is well known
and its straightforward proof can safely be left to the reader.

\begin{lemma}
\label{degree_formula}
Let $\fol$ be a Riccati foliation on $F_1$ and denote by $\fol_{\Pp}$ the corresponding Riccati foliation on
$\CP^2$. Assume that $\fol$ has exactly $k$ invariant fibers, all of them being simple.
Then the degree of $\fol_{\Pp}$ as a foliation on $\CP^2$ equals $k-1$.\qed
\end{lemma}

\section{Examples of Riccati foliations and Theorem~B}\label{examplesetc}

The purpose of this section is to detail two applications of Proposition~\ref{Riccati_foliation-F1}.
First, we will show how examples of foliations of degree~$2$ on $\CP^2$ satisfying
the conditions of Theorem~A with real analytic Levi-flats can be obtained. Similarly, we will also construct foliations
of degree~$3$ as in Theorem~A having transversely fractal Levi-flats.
Then we will turn to Theorem~B and apply
again Proposition~\ref{Riccati_foliation-F1} to construct the corresponding foliations.

\subsection{Two examples of Riccati foliations having degrees~2 or~3}

Let us start with {\it triangular groups} or more precisely, with triangular groups {\it without cusps}. This is as follows.
Let $m_i \in \Z_+^{\ast}$, $i=1,2,3$, be three positive integers satisfying the condition
\begin{equation}
\frac{1}{m_1} + \frac{1}{m_2} + \frac{1}{m_3} < 1 \, , \label{hyperbolictriangulargroups_1}
\end{equation}
and set $\textsl{m} = (m_1,m_2,m_3)$.
It is well know that, up to conjugation, there is exactly one triplet $(\xi_1, \xi_2, \xi_3)$ of elements
in $\psl$ verifying the relations
\begin{equation}
\xi_1 \xi_2 \xi_3 = \xi_1^{m_1} = \xi_2^{m_2}= \xi_3^{m_3} = {\rm id} \, . \label{relations-defininggamma}
\end{equation}
The subgroup $\Gamma$ of $\psl$ generated by $\xi_1$, $\xi_2$, and $\xi_3$ is, therefore, unique up to inner automorphisms
of $\psl$. The group $\Gamma$ is called the {\it triangular group associated
with the unordered triplet $\textsl{m} = (m_1,m_2,m_3)$}. These triangular groups happen
to be discrete and conjugate to a subgroup of $\pslR$. In other words, triangular groups $\Gamma$
as above are Fuchsian groups. Clearly, the quotient of the hyperbolic disc $\D$ by $\Gamma$ is naturally a
{\it spherical orbifold}\, with three singular points, see for example \cite{svetlana}.

\begin{ex}\label{Riccatifortriangulargroup}
Choose $\textsl{m} = (m_1,m_2,m_3)$ with $m_i \in \Z_+^{\ast}$, $i=1,2,3$, satisfying
condition~(\ref{hyperbolictriangulargroups_1})
and denote by $\fol^{(\textsl{m})}$ the Riccati foliation on the surface $F_1$
obtained by applying Proposition~\ref{Riccati_foliation-F1} to the group $\Gamma$
with generating set $\{ \xi_1, \xi_2, \xi_3 \}$.

The first family of Riccati foliations on $\CP^2$ to be considered here is, therefore,
$\fol^{(\textsl{m})}_{\Pp}$. In other words, for each fixed $\textsl{m}$, $\fol^{(\textsl{m})}_{\Pp}$ is the
foliation obtained from $\fol^{(\textsl{m})}$ by collapsing the $(-1)$-curve in the surface $F_1$.

Note that the entire family of foliations $\fol^{(\textsl{m})}_{\Pp}$ is constituted by foliations of degree~$2$
as it follows from the combination of Remark~\ref{summarizingpoints_proposition} and Lemma~\ref{degree_formula}.
Alternatively, we can consider the $3$-parameter family
\begin{eqnarray*}
	X_{(\alpha_1, \alpha_2, \alpha_3)} & = &  [\alpha_1z_1^2 + (1-\alpha_1)(z_1z_2 + z_1z_3 - z_2z_3)] \frac{\partial}{\partial z_1}  \\
	& & \, + \,  [\alpha_2z_2^2 + (1-\alpha_2)(z_1z_2 - z_1z_3 + z_2z_3)] \frac{\partial}{\partial z_2}  \\
	& & \, +  \, [\alpha_3z_3^2 + (1-\alpha_3)(-z_1z_2 + z_1z_3 + z_2z_3)] \frac{\partial}{\partial z_3} \, 
\end{eqnarray*}
of quadratic vector fields on $\C^3$ which corresponds to the family of
Halphen vector fields thoroughly studied by A. Guillot in \cite{adolfo-IHES}. In particular, the existence of a Levi-flat
for the foliation $\fol^{(\textsl{m})}_{\Pp}$ is already pointed out in \cite{adolfo-IHES}.
As a family of homogeneous vector fields of degree~$2$ on $\C^3$, the foliations they induce on $\CP^2$ are all
of degree~$2$. To obtain the foliation $\fol^{(\textsl{m})}_{\Pp}$ it suffices to choose the parameters
$\alpha_1$, $\alpha_2$, and $\alpha_3$ so that $m_i = (\alpha_1 + \alpha_2 + \alpha_3 -2) /\alpha_i$.

\begin{remark}\label{OnHalphenPSL2Z}
A case of Halphen vector fields also discussed in \cite{adolfo-IHES} and escaping the setting of Theorem~A occurs
for $\alpha_1 =-6$, $\alpha_2 =-4$, and $\alpha_3 =0$. In this case, the resulting triangular group acquires a cusp
and, in fact, is isomorphic to $\pslZ$. In other words, $\xi_1$ and $\xi_2$ are elliptic of orders~$2$ and~$3$
but $\xi_3$ {\it is parabolic}. This example still possesses a singular Levi-flat whose structure is, however,
different from what is described in Theorem~A.
Also the description of the corresponding positive foliated harmonic currents requires significantly different arguments
and will be discussed in a future work.
\end{remark}
\end{ex}

We can now proceed with an example of Riccati foliation of degree~$3$ on $\CP^2$ giving rise to a transversely
fractal Levi-flat.

\begin{ex}\label{Riccati_quasifuchsian}
In this example, we begin by considering a spherical orbifold with $4$ singular points (of orbifold type).
We then consider the Fuchsian group $\Gamma'$ arising from this orbifold by means of
Riemann uniformisation theorem.
In slightly more concrete terms, we can use Poincar\'e theorem to ensure the existence of these orbifolds.
In fact, if $m_1, m_2, m_3$ and $m_4$ are the orders of the orbifold-type singular points, then the condition for
the existence of the desired Fuchsian group $\Gamma'$ becomes (see for example \cite{svetlana})
\begin{equation}
\frac{1}{m_1} + \frac{1}{m_2} + \frac{1}{m_3} + \frac{1}{m_4}< 2. \label{justonequotation}
\end{equation}
The construction of one such Fuchsian group is simple enough to be recalled. Consider four radii issued from
$0 \in \D \subset \C$ with adjacent radii defining an angle of $\pi/2$ (for example $\R_+$, $i \R_+$, $\R_-$ and
$i \R_-$). Now for $\tau \in (0,1)$, choose a point in each radii such that its Euclidean distance to
$0 \in \D \subset \C$ is $\tau$. Next define a quadrilateral $\mathcal{Q}$ by joining points lying in adjacent sides by
segments of hyperbolic geodesic. The sides of $\mathcal{Q}$ are denoted by $l_1, l_2, l_3$ and $l_4$. Then on top of each side
$l_i$ we construct an isosceles triangle whose angle opposed to the side $l_i$, i.e., the angle defined by
the two sides of the triangle having the same length, equals $2\pi/m_i$. The result of this construction is,
therefore, a star-shaped hyperbolic octagon whose sides are labeled $\xi_1, \xi_1', \ldots ,
\xi_4, \xi_4'$ (for the orientation, see figure 17 on page 93 of \cite{svetlana}).
Furthermore, by considering the asymptotic values of the area of this octagon when $\tau \rightarrow 0$ and
$\tau \rightarrow 1$, it becomes clear that $\tau$ can be chosen
so that the area in question equals
\begin{equation}
2\pi \left[ 2 - \sum_{i=1}^4 \frac{1}{m_i} \right] \, . \label{sumlessthan2}
\end{equation}
By construction, the sides $\xi_i$ and $\xi_i'$ have the same length ($i=1,2,3,4$). In particular, there
are orientation-preserving automorphisms $H_i$ of $\D$ such that $H_i (\xi_i') = \xi_i$. Let then
$\Gamma' \subset {\rm Aut}\, (\D) \simeq \pslR \subset \psl$ generated by $H_1, H_2, H_3$ and $H_4$. The conditions
of Poincar\'e's polygon theorem are, hence, matched so that $\Gamma'$ is discrete, and, hence, Fuchsian,
moreover the quotient of $\D$ by $\Gamma'$ identifies with the surface (orbifold) obtaining by following
the above indicated gluing of the sides of the octagon. It is straightforward checking that this quotient
is topologically a sphere with four marked (singular) points corresponding to the vertices of the octagon
arising from the isosceles triangles (the vertices with angles $2\pi/m_i$). This completes the construction
of a Fuchsian group $\Gamma'$ with the desired properties.

The Kleinian group $\Gamma$ desired for this example is not the Fuchsian group $\Gamma'$ but rather some
{\it quasiconformal}\, deformation of $\Gamma'$. Recall that a discrete subgroup of $\psl$ is said to be
quasifuchsian if it leaves invariant a Jordan curve $\mathcal{J} \subset \CP^1$ which is not
a real projective circle (in which case the group would be Fuchsian). In particular, the curve
$\mathcal{J}$ will be nowhere differentiable and of Hausdorff dimension strictly greater than~$1$
owing to \cite{bowen}.

Now, whereas the theory of quasiconformal deformations and Bers Simultaneous Uniformisation
theorem are better known in the case of Riemann surfaces, similar statements still hold in the case of Riemann
surface orbifolds. The reader is referred to \cite{nag}, Section~$4.4$, for the deformation theory of Fuchsian
groups containing elliptic elements. In particular, there follows that the Teichm\"{u}ller space
of these spherical orbifolds with four singular points, or equivalently the Teichm\"{u}ller space of a Fuchsian
group as $\Gamma'$ above, is of complex dimension~$1$.

Let then $R$ denote the spherical orbifold arising from the Fuchsian group $\Gamma'$. Bers Simultaneous Uniformisation
theorem assigns to each point $p$ in the Teichm\"{u}ller space of $R$ a quasifuchsian group
$\Gamma$, unique up to conjugation, satisfying the following conditions:
\begin{itemize}
	\item The group $\Gamma$ leaves invariant a Jordan curve $\mathcal{J}$. The connected components of $\CP^1 \setminus
	\mathcal{J}$ are then denoted by $D_1$ and $D_2$.
	
	\item The quotient $D_1 /\Gamma$ is isomorphic to the orbifold $R$ whereas the quotient $D_2 /\Gamma$
	is isomorphic to the spherical orbifold parametrised by the point $p$.
\end{itemize}

Summarising what precedes, fix four positive integers $m_1, m_2, m_3$ and $m_4$ satisfying estimate~(\ref{justonequotation}),
there exists a complex $1$ parameter family of quasifuchsian groups $\Gamma$, pairwise non-conjugate,
and parametrising spherical orbifolds with four singular points whose orders are precisely $m_1, m_2, m_3$ and $m_4$.

Finally, for our example of Riccati foliations with transversely fractal Levi-flats, we consider
the family of Riccati foliations $\fol^{\mathcal{Q}}$ on $F_1$ arising from the above family of
Fuchsian groups by means of Proposition~\ref{Riccati_foliation-F1}. Then the desired examples arise as
the corresponding Riccati foliations $\fol^{\mathcal{Q}}_{\Pp}$ on $\CP^2$ obtained by collapsing the $(-1)$-rational
curve in $F_1$. In particular, the global holonomy of the foliations $\fol^{\mathcal{Q}}_{\Pp}$ is
given by the mentioned quasifuchsien groups. The reader will also notice that all the foliations
$\fol^{\mathcal{Q}}_{\Pp}$ {\it are of degree~$3$}\, as follows from Lemma~\ref{degree_formula} since they
have exactly four invariant lines, all of them being simple.
\end{ex}

\subsection{High degree foliations with diffuse positive foliated closed currents}\label{ahlforscurrent_examples} In this section
we will prove Theorem~B by building on some of the previous ideas related to Riccati foliations and
to Proposition~\ref{Riccati_foliation-F1}. In the course of the section, we will also accurately formulate
several notions involving currents and foliations appearing both in Theorems~A and~B.

Let us begin with some basic notions about currents on a compact complex surface $M$. Let $\mathcal{D}^k$ denote the
space of (smooth) differential forms of degree~$k$ on $M$. The space $\mathcal{D}_k'$ of
{\it currents of dimension~$k$}\, (where $0 \leq k \leq 4$)
is the $C^{\infty}$-topological dual of $\mathcal{D}^k$. The {\it degree}\, 
of a current $T$ is also defined as the difference between the real dimension of $M$ and the dimension of $T$.
In local coordinates, a current $T$ of dimension~$k$ acts as a $(4-k)$-form on the coefficients of a $k$-differential
form $\omega$. In fact, a current $T$ of dimension~$k$ can be represented
as a $(4-k)$-differential form with distributional coefficients.

The exterior differential operator $d$, as well as the standard operators $\partial$ and $\overline{\partial}$,
are defined on currents by letting $\langle dT, \omega \rangle = (-1)^{k'+1} \langle T, d\omega \rangle$, where
$T$ is a current of degree~$k'$ and $\omega$ a $(3-k')$-differential form. The action of $\partial$ and
$\overline{\partial}$ on currents
is analogously defined. In particular, the space of currents can be stratified with bidegrees~$p,q$. Also,
a current $T$ will be called {\it closed}\, if $dT=0$ and it will be called {\it harmonic}\, if
$i\partial \overline{\partial} T =0$.

The complex dimension of $M$ being~$2$, we will particularly be interested in $(1,1)$-currents.
Recall that a $(1,1)$-form is said to be {\it weakly positive}\, if locally it takes on the form
$$
\sum_{i,j=1}^2 \alpha_{i,j} dz_i \wedge d\overline{z}_j,
$$
with the matrix $\{ \alpha_{i,j} \}$ being positive semi-definite at every point. In turn, a $(1,1)$-current is said
to be {\it positive}\, if the coupling $T \wedge \omega$ is a positive measure for every weakly positive form
$\omega$.

Assume now that $M$ is endowed with a (singular) holomorphic foliation $\fol$. Then, among all currents on $M$,
we may look for those providing special insight in the structure of the foliation $\fol$. This gives rise to the notion
of {\it foliated}\, or {\it directed}\, current.

\begin{defi}
	Let $M$ be a complex compact surface equipped with a singular holomorphic foliation $\fol$. A
	current $T$ on $M$ is said to be foliated, or directed by $\fol$, if we have $T (\beta) = 0$ whenever
	$\beta$ is a $2$-form vanishing on the tangent space of the foliation $\fol$.
\end{defi}

In foliated coordinates $(z_1,z_2)$ where the foliation $\fol$ is given by $\{ dz_2 = 0\}$,
every $2$-form $\beta$ vanishing on the tangent
space of the foliation $\fol$ must have the form $\alpha_1 dz_2 + \alpha_2 d\overline{z}_2$ for suitable
$1$-forms $\alpha_1$, $\alpha_2$. From this, we see that a $(1,1)$-foliated current $T$ must take on
the local form
$$
T = T(z_1,z_2) \; dz_2 \wedge d\overline{z}_2 \, ,
$$
where $T$ is identified with a distribution. In particular, $T$ is of type~$(1,1)$. Moreover, in the context of foliated
currents, the notion of being {\it a positive current}\, becomes
particularly clear, it simply means that {\it $T(z_1,z_2)$ is identified with a positive measure}.

Consider now the positive foliated form $T = T(z_1,z_2) dz_2 \wedge d\overline{z}_2$.
If, in addition, $T$ is {\it closed}, then the distributional derivative of $T(z_1,z_2)$ with respect to
$z_1$ and $\overline{z}_1$
must vanish. This means that $T(z_1,z_2)$ is essentially constant over the plaques of the foliation $\fol$ or,
equivalently, that $T(z_1,z_2)$ depends only on~$z_2$ and, hence, it naturally induces a positive measure
on the transverse space (a particular case of Rokhlin disintegration).

Similarly, if $T$ is {\it harmonic}, then $T(z_1,z_2)$ yields a harmonic function over the plaques of the foliation $\fol$
which is bounded above and below by strictly positive constants. By Rokhlin disintegration,
$T$ still defines a measure on the transverse space which is referred to as the {\it harmonic measure}.

According to Sullivan \cite{sullivan}, in the case of a regular foliation defined on a compact manifold,
positive foliated closed currents are in one-to-one correspondence with {\it transversely invariant measures}
as defined by Plante \cite{plante} and recalled below.

\begin{defi}\label{Plante_definition}
	Given a regular foliation $\fol$ on a compact complex manifold $M$, a transversely invariant measure for the foliation $\fol$
	consists of the following data:
	\begin{itemize}
        \item[(1)] A (finite) foliated atlas $\{ (U_i, \varphi_i, \Sigma_i)\}$, where $\Sigma_i$ stands for the transverse
        space.
        
        \item[(2)] For every $i$, a finite measure $\mu_i$ defined on the transverse space $\Sigma_i$.
        
        \item[(3)] Whenever $U_i \cap U_j \neq \emptyset$, the ``gluing map'' $h_{ij}$ defined by
        $\varphi_j \circ \varphi_i^{-1} (z_1,z_2) = (f(z_1,z_2),h_{ij} (z_2))$ satisfies
        $h_{ij}^{\ast} \mu_j = \mu_i$ whenever both sides are defined.

    \end{itemize}
\end{defi}

In the case of a singular foliation, we simply repeat the above definition on the open manifold
$M \setminus {\rm Sing}\, (\fol)$ while dropping the condition on finiteness of the foliated
atlas $\{ (U_i, \varphi_i, \Sigma_i)\}$. It is then clear that a (positive foliated) closed current
still yields a transversely invariant measure for the foliation $\fol$ in the sense of
Definition~\ref{Plante_definition}.
Only the converse, i.e., the building of such a current out of a given transversely invariant measure,
needs further attention.

We are now ready to begin the proof of Theorem~B. Fix $n \in \N$ and consider $\CP^1$ with $(n+1)$ marked
points $b_1, \ldots , b_{n+1}$. Fix also an elliptic element $\xi \in \psl$ representing an irrational
rotation and let $\Gamma$ be the subgroup generated by $\xi$. Let $\widetilde{\fol}_n$ be the
Riccati foliation on $F_1$ arising from Proposition~\ref{Riccati_foliation-F1} with $(n+1)$ simple invariant fibers
$C_1, \ldots , C_{n+1}$. More precisely,
the local holonomy map of the foliation $\widetilde{\fol}_n$ around each of the invariant fibers $C_1, \ldots ,C_n$
coincides with $\xi$ whereas the local holonomy map obtained by winding around
$C_{n+1}$ coincides with $\xi^{-n}$. Moreover, each invariant fiber
$C_i$ of the foliation $\widetilde{\fol}_n$ contains exactly two singular points of this foliation, namely:
one point denoted by $p_i$ where the quotient of the eigenvalues of the foliation
$\widetilde{\fol}_n$ is positive real (a``sink'') and a point denoted by $q_i$ where this quotient
is negative real (a``saddle''). Thus, the foliation $\widetilde{\fol}_n$ possesses~$2n +2$ singular points
and all of them are simple. It follows from the local models for the foliation $\widetilde{\fol}_n$ around the invariant
fibers $C_i$ described in the proof of Proposition~\ref{Riccati_foliation-F1}
that each point $p_i$ (resp. $q_i$) admits exactly two separatrices. Furthermore, one of these
separatrices coincides with the invariant fiber $C_i$ while the other separatrix,
which is also smooth, happens to be transverse to the invariant fiber $C_i$.
In the sequel, we let $S_{p_i}$ (resp. $S_{q_i}$) denote the
separatrix of $\widetilde{\fol}_n$ at the singular point $p_i$ (resp. $q_i$) that is transverse
to~$C_i$.

\begin{lemma} \label{rational curves}
Without loss of generality, we can assume that all the separatrices $S_{p_i}$ (resp. $S_{q_i}$) glue together
into a global rational curve of self-intersection~$1$
(resp. $-1$) invariant by the foliation $\widetilde{\fol}_n$.
\end{lemma}

\begin{proof}
Recall that the global holonomy group $\Gamma$ of the foliation $\widetilde{\fol}_n$ is generated by $\xi$
(i.e., $\Gamma = \langle \xi \rangle$) and that, up to conjugation, the action of $\Gamma$ on $\CP^1$ can
be represented in any fiber of the surface $F_1$ different from the invariant fibers $C_1, \ldots, C_{n+1}$.

Let $C_i$ be a invariant fiber and consider a generic (non-invariant) fiber $\Sigma$ near to $C_i$ in the following
sense: The local separatrices $S_{p_i}$, $S_{q_i}$ associated with the singular points $p_i$, $q_i$ intersect
$\Sigma$ transversely at well defined points $x_{p_i}$ and $x_{q_i}$. Clearly, the points $x_{p_i}$, $x_{q_i}$
are fixed by the automorphism of $\Sigma$ obtained as the local holonomy around $C_i$. This construction,
holds for every $i=1, \ldots, n+1$. Now, once by parallel transport, we represent all of these points in the same
fiber $\Sigma$, they must coincide (up to order) into two well defined set points which, incidentally,
are the fixed points of (a conjugate of) $\xi$. Indeed, all
the local holonomy maps around invariant fibers coincide (strictly speaking except for the fiber
$C_{n+1}$ which coincides with the
$n^{\rm th}$-power of the previous ones). In other words, given any two invariant fibers
$C_i$ and $C_j$, either we have $x_{p_i}=x_{p_j}$ and $x_{q_i}=x_{q_j}$ or we have
$x_{p_i}=x_{q_j}$ and $x_{q_i}=x_{p_j}$. However, we can assume without loss of generality that the
first possibility holds. Indeed, this happens thanks to the local models for the foliation $\widetilde{\fol}_n$
around the invariant fibers $C_i$ described in the proof of Proposition~\ref{Riccati_foliation-F1}: the position of
the singular points can be permuted without affecting the holonomy.
Thus, we can ensure that all the separatrices $S_{p_i}$, $i=1, \ldots, n+1$, glue together so as to define a
invariant compact curve $R_p$. Similarly, the separatrices $S_{q_i}$, $i=1, \ldots, n+1$, yield another
invariant compact curve denoted by $R_q$. Due to the transverse nature of the Riccati foliation with respect to
the fibers of $F_1$, it is immediate to conclude that both $R_p$ and $R_q$ are sections of $F_1$.
However, since the eigenvalues of the foliation $\widetilde{\fol}_d$ at the singular points $p_i$
(resp. $q_i$) are positive (resp. negative) real numbers, there follows from Camacho-Sad index theorem \cite{camacho}
that $R_p$ (resp. $R_q$) has positive (resp. negative) self-intersection. Finally, it is well known
that the surface $F_1$ possesses only two sections which, incidentally, are rational curves of
self-intersection~$1$ and~$-1$, see \cite{barth}. The lemma follows.
\end{proof}

Again let $\Sigma$ be a generic fiber of $F_1$. The generator $\xi$ of the global holonomy group $\Gamma$
of the foliation $\widetilde{\fol}_n$ being an elliptic element of infinite order, there exists a continuum
of circles contained in $\Sigma$ that are invariant under $\xi$. Fix one of these invariant circles
and denote it by $S^1$. Note that
the normalised Lebesgue measure on $S^1$ is invariant under $\xi$.
Furthermore, the points $x_{p}$ and $x_{q}$ where the section $R_p$ and $R_q$ intersect $\Sigma$
are the fixed points of $\xi$ and, therefore, lie away from $S^1$.

Let $\overline{\fol \, (S^1)}$ be the closure of the saturated set $\fol(S^1)$ of the invariant circle $S^1$ by the
foliation $\widetilde{\fol}_n$. We have first to show that $C_i \cap \overline{\fol \, (S^1)} = \{ p_i \}$. To do so,
let us begin with the following claim: $C_i \cap \overline{\fol \, (S^1)} \subset \{p_i,q_i\}$. To prove the claim
we argue by contradiction. Assume it is false. Then
$\fol(S^1)$ accumulates on a regular point of the invariant fiber $C_i$. Since the foliation $\widetilde{\fol}_n$
has a (linearisable) saddle singular point at $q_i$, the fact that $\fol(S^1)$ accumulates on a regular point of
the invariant fiber $C_i$ implies that $\fol(S^1)$ must accumulate at regular points of the separatrix $S_{q_i}$
as well. However, the separatrix $S_{q_i}$ intersects the generic fiber $\Sigma$ at a fixed point of $\xi$ that,
in turn, is accumulated by points in $\fol(S^1)$. However, this is
impossible since the intersection of $\fol(S^1)$ and $\Sigma$
is the initial invariant circle $S^1$ which lies away from the fixed points of $\xi$. The resulting contradiction
then establishes the claim.

Next, note that the above argument actually
shows that the separatrix $S_{q_i}$ cannot be contained in $\overline{\fol \, (S^1)}$. Thus, $\fol \, (S^1)$ cannot
accumulate on any of the two separatrices of the foliation $\widetilde{\fol}_n$ at the singular point $q_i$. Since
the singular point $q_i$ lies in the Siegel domain, this implies that $\fol \, (S^1)$ cannot accumulate on the singular
point~$q_i$ itself. Hence, the intersection $C_i \cap \overline{\fol \, (S^1)}$ is reduced to $\{ p_i \}$
since it cannot be empty. We have then proved that $C_i \cap \overline{\fol \, (S^1)} = \{ p_i \}$ as required.

For each $i \in \{1, \ldots ,n+1\}$, let $B_i \subset F_1$ be a small neighbourhood around the singular point $p_i$
along with coordinates $(z^i_1,z^i_2)$ where the restriction of the foliation $\widetilde{\fol}_n$ to $B_i$ is
linear, i.e., the foliation $\widetilde{\fol}_n$ is locally given by a vector field having the
form
$$
z^i_1 \frac{\partial}{\partial z^i_1}+ \lambda z^i_2 \frac{\partial}{\partial z^i_2},
$$
where $\lambda \in \R^+ \setminus \Q$. Similarly, for each $B_i$, $i=1, \ldots , d+1$, let $V_i \subset B_i$
be a (smaller) neighborhood of $p_i$ properly contained in $B_i$.
The foliation $\widetilde{\fol}_n$ being regular on the complement of $\bigcup V_i$, consider a covering
of $F_1 \setminus \bigcup V_i$ by foliated coordinates $\{(U_j, \varphi_j, \Sigma_j)\}_{j=1}^N$.
In fact, in what follows, it would be enough to consider the restriction
of $\widetilde{\fol}_n$ to $\overline{\fol \, (S^1)}$ as a (singular) lamination
by Riemann surfaces, see for example \cite{fornaesssibony-survey}.

Before continuing, recall that a positive foliated closed current is said to be {\it algebraic}\, if it
coincides with a constant multiple of the current given by integration over a (possibly singular)
algebraic invariant curve. Otherwise, the current is said to {\it difuse}. Similarly, given a (non-constant)
holomorphic map $f : \C \rightarrow \CP^2$, an Ahlfors current is a positive {\it closed} current
of dimension~$(1,1)$ which is the limit of a sequence
$$
\frac{f_{\ast}[D_{r_j}]}{{\rm Area}\, (f(D_{r_j}))}
$$
where $r_j \rightarrow \infty$. Here $[D_{r_j}]$ stands for the integration current
over the disc $D_{r_j} \subset \C$ of radius $r_j$ and the {\it area} 
${\rm Area}\, (f(D_{r_j}))$ is nothing but the area of $f(D_{r_j})$ computed with respect to
any auxiliary Riemannian metric. Alternatively, 
${\rm Area}\, (f(D_{r_j}))$ can be defined as the integral over $D_{r_j}$ of the pull-back
by $f$ of the standard Fubini-Study form on $\CP^2$.

In what follows, a $(1,1)$-positive foliated closed current $\widetilde{T}$ on $F_1$ that, in addition, is diffuse and
of Ahlfors type will be constructed. Basically the construction will be carried out two steps due to
the presence of singular points. To begin, consider
a $(1,1)$-differential form $\omega$ on $F_1$ and let it be decomposed as $\omega=\omega_1+\omega_2$
where $\omega_1$ and $\omega_2$ are $(1,1)$-differential
forms supported on $F_1 \setminus \bigcup V_i$ and $\bigcup B_i$, respectively.

First, we are going to define $\widetilde{T}$ for forms like $\omega_1$, i.e., forms whose supports are contained in
$F_1 \setminus \bigcup V_i$. This is very much the general construction in \cite{sullivan} since singular
points play no role. Indeed, to define $\widetilde{T}$ in this case we proceed as follows. Since the singular points
of the foliation $\widetilde{\fol}_n$ lie away from the support of $\omega_1$ fix a partition of the unity
$(a_j)_{j=1}^N$ strictly subordinate to the cover $(U_j)_{j=1}^N$ of $F_1 \setminus \bigcup V_i$. The
product $a_j \omega_1$ is identified with a continuous function on $\Sigma_j$. In fact, we can consider a
function $f_j:\Sigma_j \rightarrow \C$ defined by
$$
z_2 \mapsto f_j(z_2)=\int_{P_{z_2}}a_j \omega_1(z_2),
$$
in terms of foliated local coordinates $(z_1,z_2)$ in $U_j$, the plaques of the restriction of $\tilf_n$
to $U_j$ are of the form $P_{z_2} \subset \C \times \{z_2\}$ while the
transverse section $\Sigma_j$ becomes $\Sigma_j \subset \{z_1=0\}$. Then we integrate and sum these functions
obtaining the value
$$
\widetilde{T}(\omega_1)=\sum_{j=1}^{N}\left( \int_{\Sigma_j}f_j(z_2)d\mu_j(z_2)\right)
=\sum_{j=1}^{N} \int_{\Sigma_j} \left( \int_{P_{z_2}} a_j \omega_1(z_2) \right) d\mu_j(z_2) \, ,
$$
where $\{ \mu_j\}$ is the transverse invariant measure induced by the normalised Lebesgue measure on $S^1$.


Now we have to define $\widetilde{T}$ for forms like $\omega_2$, i.e., those whose supports are
contained in $\bigcup B_i$.
Without loss of generality, we can assume that the support of $\omega_2$ is contained in $B_1$.
Consider the restriction $\widetilde{\fol}_{\vert B_1}$
of $\widetilde{\fol}_n$ to $B_1$ and let $D_1$ be a transverse section
to this foliation. The invariant circle $S^1 \subset \Sigma$ induces a circle $S^1_1 \subset D_1$ which is invariant under
the holonomy of $\widetilde{\fol}_{\vert B_1}$. Next, define
$\Phi: \mathcal{S} \times S^1_1 \rightarrow B_1$ by
$$
\Phi(t,y)=(e^t,ye^{\lambda t}),
$$
where $\mathcal{S}= \{t\in \C: {\rm Re}(t) \leq 0, 0 \leq {\rm Im}(t) \leq 2\pi\}$. The leaves of the
foliation $\widetilde{\fol}_{\vert B_1}$ are then parametrised by $\Phi$. More precisely, for
$y \in S^1_1$ fixed, let $\Phi_{y}: \mathcal{S} \rightarrow B_1$ be defined by $\Phi_{y}(t)=\Phi(t,y)$.
Then set $L_{y}=\{\Phi_{y}(t): t \in \mathcal{S} \}$. Clearly each $L_y$ is contained in a leaf of 
$\widetilde{\fol}_{\vert B_1}$ and the union of these pieces of leaves over an orbit of the holonomy group
yields all of the leaf in question. Hence, to have $\widetilde{T}$ well defined, it suffices
to show that the (improper) double integral
\begin{equation}
\int_{S^1_1} \left( \int_{L_y} \omega_2 \right) d\mu_1(y) \label{doubleintegral_closedcurrents}
\end{equation}
converges. In turn, owing to Dominated convergence and to Tonelli theorem, this double integral converges provided that
the improper integral
$$
\int_{L_y} \omega_2 = \int_{\mathcal{S}} \Phi_y^*\omega_2
$$
is uniformly bounded (independent of $y \in S^1_1$). Let us then show that the latter integral is, indeed, uniformly
bounded on~$y$.

To do so, recall first that $B_1$ is equipped with coordinates $(z^1_1, z^1_2)$ where the foliation
$\widetilde{\fol}_n$ is given by the vector field $z^1_1 \partial /\partial z^1_1 + 
\lambda z^1_2 \partial /\partial z^1_2$ with $\lambda \in \R^+ \setminus \Q$ irrational.
In particular, $L_y$ is parametrised by $\Phi_{y}(t) = (e^t,ye^{\lambda t}) = (z^1_1,z^1_2)$, $t \in \mathcal{S}$.
(where $y \in S^1_1$).
Hence
$dz^1_1=e^tdt$, $dz^1_2=\lambda ye^{\lambda t}dt$, $d\overline{z}^1_1=e^{\overline{t}}d\overline{t}$,
and $d\overline{z}^1_2=\lambda \overline{y}e^{\lambda \overline{t}}d\overline{t}$. 
Letting
$$
\omega_2= \sum_{s,k=1}^{2}a_{sk}dz^1_s \wedge d\overline{z}^1_k,
$$
we conclude that
$$
\Phi_y^*\omega_2=\left(e^te^{\overline{t}}a_{11}+\lambda \overline{y}e^te^{\lambda \overline{t}}a_{12}+\lambda ye^{\overline{t}}e^{\lambda t}a_{21}+\lambda^2 \vert y \vert^2e^{\lambda t}e^{\lambda \overline{t}}a_{22}\right)dt
\wedge d\overline{t} = J(t,y) dt \wedge d\overline{t} \, .
$$
Where, in slightly more explicit terms, $J(t,y)$ is given by
$$
e^{2{\rm Re}(t)}a_{11} \! + \! \lambda \overline{y}e^{(\lambda +1){\rm Re}(t)+i(-\lambda +1){\rm Im}(t)}a_{12} \!+\!
\lambda ye^{(\lambda +1){\rm Re}(t)+i(\lambda -1){\rm Im}(t)}a_{21}\! + \!\lambda^2
\vert y \vert^2e^{2 \lambda {\rm Re}(t)}a_{22} .
$$
Since $\lambda$ is positive real, there follows that the integrals over $\mathcal{S}$ of each of the functions
$e^{2{\rm Re}(t)}$, $e^{(\lambda +1){\rm Re}(t)+i(-\lambda +1){\rm Im}(t)}$, 
$e^{(\lambda +1){\rm Re}(t)+i(\lambda -1){\rm Im}(t)}$, and $e^{2 \lambda {\rm Re}(t)}$ are all absolutely convergent.
Furthermore, $\vert y\vert$ is also uniformly bounded since $y \in S^1_1$ and all the coefficients
$a_{sk}$ are also bounded since they are identified with $C^{\infty}$ functions supported in $B_1$.
Thus the integral in~(\ref{doubleintegral_closedcurrents}) is convergent. Hence the coupling of
$\widetilde{T}$ and $\omega_2$ given by
$$
\widetilde{T}(\omega_2)=\sum_{j=1}^{d+1} \int_{S^1_j} \left( \int_{L_y} \omega_2 \right) d\mu_j(y) \, 
$$
is well defined on forms $\omega_2$ whose support is contained in the union of the sets $B_i$, $i=1,\ldots, d+1$.
Now, the formula
$$
\widetilde{T}(\omega)=\widetilde{T}(\omega_1+\omega_2)=\sum_{s=1}^{N} \int_{\Sigma_s}
\left( \int_{P_{z_2}} a_s \omega_1(z_2) \right) d\mu_s(z_2)+\sum_{k=1}^{d+1} \int_{S^1_k}
\left( \int_{L_y} \omega_2 \right) d\mu_k(y)
$$
yields a well defined current on $F_1$ since it clearly does not depend on the decomposition $\omega=\omega_1+\omega_2$.
Clearly $\widetilde{T}$ is positive, foliated, and closed.

\begin{proof}[Proof of Theorem~B]
In the preceding, we have constructed a $(1,1)$-diffuse positive foliated closed current $\widetilde{T}$
on the surface $F_1$. It only remains to check that the current $\widetilde{T}$ actually is of Ahlfors type. For this,
notice first any leaf $L$ of the foliation $\widetilde{\fol}_n$
contained in $\overline{\fol \, (S^1)}$ is dense in $\overline{\fol \, (S^1)}$ since $\gamma$ is an irrational rotation.
Consider then a leaf $L_{0}$ of the foliation $\widetilde{\fol}_n$ contained in $\overline{\fol \, (S^1)}$.
Note that $L_0$ goes through the singular points $p_i$, $i=1, \ldots, n+1$ with infinitely many branches. Still,
away from these singular points $p_1, \ldots , p_{n+1}$, the growth type of $L_0$ 
is determined by the growth of the global holonomy group $\Gamma$ of the
foliation $\widetilde{\fol}_d$. The group $\Gamma$ being cyclic, its growth is linear.
However, the combination of the previously considered parametrisation of
the leaves of $\widetilde{\fol}_n$ in the sets $B_i$ with the linear growth of the holonomy group
$\Gamma$ guarantees the existence of exhaustion $(D_{r_j})_{j \in \N}$ of the leaf $L_0$ such that
$$
\lim_{j \to +\infty} \frac{{\rm Length}(\partial D_{r_j})}{{\rm Area} (D_{r_j})}=0,
$$
where $\partial D_{r_j}$ stands for the boundary of $D_{r_j}$ and where length and area are computed with respect
to an auxiliary Hermitian metric fixed on $F_1$. In particular, as Riemann surface, the leaf $L_0$ is a
quotient of $\C$. Similarly, an Ahlfors current $T_A$ supported on $\overline{\fol \, (S^1)}$ can be obtained
as an accumulation point for the sequence of normalized integration currents associated with the exhaustion
$(D_{r_j})_{j \in \N}$. Since all currents in the mentioned sequence
are clearly foliated, so it will be its accumulation points.
Now, as a foliated
closed current, $T_A$ can be disintegrated on $\Sigma$ to yield a measure on $S^1$ that is invariant under $\xi$.
However, the action of $\xi$ on $S^1$ being conjugate to an irrational rotation, it admits a unique invariant
probability measure thanks to the well-known Weyl's theorem. Hence, up to a constant multiple, $T_A$ and
$\widetilde{T}$ induce the same measure on $S^1 \subset \Sigma$. Therefore, they themselves coincide up
to multiplication by a positive constant.

To complete the proof of Theorem~B we have to adapt the preceding construction conducted on $F_1$ to the
complex projective plane $\CP^2$.
Consider the collapsing of the~$(-1)$-rational curve in $F_1$ leading to the blow-down projection
$\pi: F_1 \rightarrow \CP^2$ and let $\fol_n$ be the resulting foliation on $\CP^2$.
By definition, $\fol_n$ is a Riccati foliation on $\CP^2$ and it has exactly $(n+1)$ invariant lines, all of
them being simple. Therefore $\fol_n$ is a degree~$n$ foliation on $\CP^2$, cf. Lemma~\ref{degree_formula}.
Naturally the push-forward of the current $\widetilde{T}$ by $\pi$ induces again a diffuse Ahlfors current $T$ for
$\fol_n$ on $\CP^2$. Finally, in terms of the singular points of $\fol_n$, recall from
Lemma~\ref{rational curves} that the $(-1)$-rational curve in $F_1$ passes through all ``saddle'' singular points
$q_i$ and avoid all ``sink'' singular points $p_i$, $i=1, \ldots, n+1$. Thus the following holds:
\begin{itemize}
	\item The $(n+1)$ simple (sink) singular points $p_i$ of the foliation $\widetilde{\fol}_n$ yield $(n+1)$ simple
	singular points for $\fol_n$.
	
	\item The $(n+1)$ simple (saddle) singular points $q_i$ of the foliation $\widetilde{\fol}_n$ sitting in the collapsed curve
	merge together in a (degenerate of order~$n$) singular point for $\fol_n$.
\end{itemize}
Away from these $n+2$ singular points, the foliation $\fol_n$ is regular as a diffeomorphic image of the regular
foliation induced by $\widetilde{\fol}_n$ on the corresponding Zariski-open set. Theorem~B is proved.
\end{proof}

Let us close this section with some examples showing that the problem of describing
diffuse positive foliated closed currents on $\CP^2$ goes beyond the description of currents that
are of Ahlfors type.

\begin{ex}\label{Hilbertfoliations}
Prototypical examples of holomorphic foliations admitting diffuse positive foliated closed currents
that are not of Ahlfors type are provided by {\it Hilbert modular foliations}.
In the sequel, we restrict ourselves to the case of complex surfaces. Let $N$ be a square-free positive integer
and denote by $K$ the totally real quadratic field $\Q (\sqrt{N})$. The ring of integers in $K$ will be denoted
by $O_K$ while $\Delta_K$ will stand for the discriminant of $K$.
The two natural embedding of $K$ in $\R$ induce an embedding of
${\rm PSL}\, (2,K)$ in ${\rm PSL}\, (2,\R) \times {\rm PSL}\, (2,\R)$.
Through this embedding, the {\it Hilbert
modular group}\, ${\rm PSL}\, (2,O_K)$ acts on two copies $\Hup \times \Hup$ of the upper half plane $\Hup$.
The quotient of this action can be compactified by adding finitely many cusps to give rise to a normal singular
compact complex space of dimension~$2$. The singularities of this space, however, are of orbifold-type and arise
from elliptic elements in ${\rm PSL}\, (2,O_K)$. Once they are resolved in a canonical minimal way, we obtain
an algebraic surface $Y(\Delta_K)$ called the {\it Hilbert modular surface of $K$}, cf. \cite{hilbert1}

Next note that the two evident foliations of $\Hup \times \Hup$ by vertical and by horizontal upper planes are both
preserved by the action of ${\rm PSL}\, (2,O_K)$. Thus the Hilbert surface $Y(\Delta_K)$ is endowed with a pair
of singular foliations, mutually transverse at a Zariski-open set, which are called {\it Hilbert modular foliations}
and are studied, in particular, in \cite{lgmendes}.
These foliations will be denoted by $\mathcal{H}$.
Zariski-dense leaves of $\mathcal{H}$, i.e., non algebraic leaves, are of hyperbolic type since they are quotients
of the upper half plane $\Hup$. In particular, $\mathcal{H}$ carries no foliated current of Ahlfors type. Yet, we have:

\vspace{0.1cm}

\noindent {\it Claim}. Every Hilbert modular foliation $\mathcal{H}$ admits a diffuse positive foliated
closed current $T$.

\begin{proof}
	Note that the transverse space to $\mathcal{H}$ on $(\Hup \times \Hup)/{\rm PSL}\, (2,O_K)$
	can naturally be identified with the quotient of $\Hup$ by the action of ${\rm PSL}\, (2,O_K) \subset
	{\rm PSL}\, (2,\R)$. From this perspective, the subgroup ${\rm PSL}\, (2,O_K)$ is actually discrete and
	the corresponding action admits a fundamental domain having finite hyperbolic volume, cf. \cite{hilbert1}.
	In turn, the hyperbolic measure is invariant under the holonomy since it locally coincides with M\"oebius
	transformation (cf. the assertion that $\mathcal{H}$ has a transverse projective structure in Theorem~1 of
	\cite{lgmendes}). In other words, the foliation $\mathcal{H}$ on the open surface
	$(\Hup \times \Hup)/{\rm PSL}\, (2,O_K)$ possesses a natural transversely invariant measure in the sense
	of Definition~\ref{Plante_definition}. We can then use this transversely invariant measure to produce
	a foliated closed current for $\mathcal{H}$ on the corresponding algebraic (compact) Hilbert surface
	$Y(\Delta_K)$ following Sullivan general theory. The fact that Sullivan's construction actually yields a well defined current
	on $Y(\Delta_K)$ can be checked by means of a straightforward adaptation of the preceding dicussion,
	keeping in mind the structure of the compactification in \cite{hilbert1}. Details are left to the reader.
\end{proof}

The above constructed current is clearly diffuse since the hyperbolic measure on $\Hup$ has no atomic part.
Finally, to produce examples defined on $\CP^2$, we note that rational Hilbert surfaces $Y(\Delta_K)$
were classified in \cite{hilbert2}. They correspond to the cases where $\Delta_K$ takes on one of the values
$5$, $8$, $12$, $13$, $17$, $21$, $24$, $28$, $33$, $60$. Therefore these surfaces are birationally equivalent
to $\CP^2$. Once again it is straightforward to follow the birational maps in question to make sure that
the initial diffuse positive foliated closed current is pushed-forward to a proper diffuse positive
foliated closed current for the corresponding foliation on $\CP^2$.
\end{ex}

\section{Riccati equations and singular Levi flats}\label{provingTheoremA_Part_I}

In this section we begin a direct approach to the proof of Theorem~A. In what follows we will provide
a general sufficient criterion for a Riccati foliation on $\CP^2$ to exhibit (singular) Levi-flats, whether
they are real analytic or only continuous with a fractal nature.
The content of the section is summarised by Theorem~\ref{a&b} which establishes some parts of Theorem~A.

Consider a Riccati foliation $\fol$ on the Hirzebruch surface $F_1$ along with a {\it simple invariant fiber $C$}. Recall
that the foliation $\fol$ has either one or two singular points in the invariant fiber $C$,
cf. Equation~(\ref{simplenormalform_Riccati_Fn}). Furthermore,
the following can easily be checked: 
\begin{itemize}
\item If the foliation $\fol$ has two singular points in the invariant fiber $C$, then at each singular point,
the eigenvalue of the foliation $\fol$ associated with the direction tangent to the invariant fiber $C$ is necessarily
different from zero. The invariant fiber $C$ being, in addition, simple, each of these singular points will have
two eigenvalues different from zero. A singular point possessing two eigenvalues $\lambda_1$ and $\lambda_2$
different from zero is said to belong to the {\it Poincar\'e domain} if the quotient $\lambda_1/\lambda_2$
lies in $\C \setminus \R_-$. If the quotient $\lambda_1/\lambda_2$ lies in $\R_-$, then the singularity is said to
belong to the {\it Siegel domain}. Finally, singularities with two non-zero eigenvalues $\lambda_1, \, \lambda_2$
such that $\lambda_1/\lambda_2 \in \C \setminus \R$ are called hyperbolic. In particular, hyperbolic singularities
belong to the Poincar\'e domain.

\item If the foliation $\fol$ has a unique singular point in the invariant fiber $C$, then the eigenvalue of
the foliation $\fol$ corresponding to the direction of the invariant fiber $C$ is equal to zero. However, the
invariant fiber $C$ being simple, the foliation $\fol$ has a non-zero eingenvalue in the direction transverse
to the invariant fiber $C$. This type of singularity is called a {\it saddle-node}. In other words,
saddle-node singularities are those that have one eigenvalue equal to zero and another different from zero.
\end{itemize}

\begin{theorem} \label{a&b}
Let $\fol$ be a Riccati foliation on $\CP^2$ as in Theorem~A. In particular, the local holonomy maps of the
foliation $\fol$ around the invariant lines are all elliptic elements. Assume that the global holonomy group
$\Gamma$ of the foliation $\fol$ is a Fuchsian (resp. quasifuchsian) group of first kind. Then the following holds:
\begin{itemize}
	\item[(a)] There exists a closed set $\mathcal{L} \subset \CP^2$ of topological dimension equal to~$3$
	which is invariant by the foliation $\fol$.
	
	\item[(b)] If $\Gamma$ is Fuchsian, then $\mathcal{L}$ is a real analytic set with $k$ singular points, all of
	them of orbifold-type. If $\Gamma$ is quasifuchsian, then $\mathcal{L}$ is a ``singular topological manifold''
	with Hausdorff dimension strictly greater than~$3$.
\end{itemize}
\end{theorem}

\begin{lemma}\label{no_saddlenodes}
Let $\fol$ be a Riccati foliation on the surface $F_1$ along with a simple invariant fiber $C$ such that the local holonomy
map of $\fol$ around the invariant fiber $C$ is not parabolic. Then the foliation $\fol$ has two singular points in the invariant
fiber $C$ and each of these singular points is associated with two eigenvalues different from zero.
\end{lemma}

\begin{proof}
Given the review at the beginning of the section about the structure of a Riccati foliation near a
simple invariant fiber, the proof amounts to checking that $\fol$ cannot have a unique singular point $p$
in the fiber $C$. Assume aiming at a contradiction that this was the case. Thus $p$ is a saddle-node singular
point for $\fol$ with a non-zero eigenvalue associated to a direction transverse to $C$. Owing to Dulac's
normal form for saddle-nodes, $\fol$ admits a separatrix $S$ at $p$ that is smooth, transverse to $C$, and tangent
to the non-zero eigenvalue of $\fol$ at~$p$ (see for example \cite{IlyashenkoYakovenko} and \cite{HelenaandI}).
The separatrix $S$ induces a fixed point for the holonomy map of $\fol$ arising from a loop around $C$. Moreover,
an elementary computation with Dulac's normal form (\cite{IlyashenkoYakovenko}, \cite{HelenaandI})
ensures the following holds:
\begin{itemize}
	\item The multiplier of the holonomy map at this fixed point equals~$1$.
	
	\item This holonomy map cannot coincide with the identity.
\end{itemize}
It follows from the preceding that the holonomy map of $\fol$ arising from winding around $C$ is a parabolic
map. The resulting contradiction proves the lemma.
\end{proof}

At this point it is convenient to remind the reader of 
some basic normal form theory for singularities of foliations
in dimension~$2$. The material mentioned below is available from most standard texts such
as \cite{IlyashenkoYakovenko}, \cite{HelenaandI}.

First consider a foliation $\fol$ defined on a neighbourhood of $(0,0) \in \C^2$ and having at
the origin two non-zero eigenvalues $\lambda_1$ and $\lambda_2$ satisfying $\lambda_1/\lambda_2 \in \R_-$.
Then there are local coordinates $(u,v)$ where the foliation $\fol$ is locally given by a
vector field having the form 
\begin{equation}
\lambda_1u[1+({\rm h.o.t})]\frac{\partial}{\partial u}+\lambda_2v[1+({\rm h.o.t})]\frac{\partial}{\partial v}. \label{Siegel}
\end{equation}   
where ${\rm h.o.t}$ stands for ``higher order terms''. In particular, it is immediate to check that the foliation
$\fol$ possesses {\it exactly}
two separatrices which, incidentally, are smooth and mutually transverse being given
in the above coordinates by the axes $\{u=0\}$ and $\{v=0\}$.

Next, let $\fol_1$ and $\fol_2$ be two holomorphic foliations as above and sharing the same (non-zero) eigenvalues
$\lambda_1$ and $\lambda_2$ ($\lambda_1/\lambda_2 \in \R_-$). Let $(u_i,v_i)$, $i=1,2$, be local coordinates
where $\fol_i$ takes on the form indicated in Equation~(\ref{Siegel}). Denote by $h_i$ the local
holonomy map of the foliation $\fol_i$ relative to the axis $\{v_i=0\}$, $i=1,2$. A theorem due
to Mattei-Moussu \cite{MatteiMoussu} then states that the foliations $\fol_1$ and $\fol_2$ are analytically
equivalent if and only if the holonomy maps $h_1$ and $h_2$ are analytically conjugate in ${\rm Diff}\, (\C,0)$.

Let us now move to the case where the two eigenvalues $\lambda_1$, $\lambda_2$ of $\fol$ at the origin
verify $\lambda_1/\lambda_2 \in \R_+$ (always with $\lambda_1 \lambda_2 \neq 0$).
First of all, if $\lambda_1$ and $\lambda_2$ {\it are non-resonant}, i.e., if neither $\lambda_1/\lambda_2$ nor
$\lambda_2/\lambda_1$ is a positive integer, then the foliation $\fol$ is {\it linearisable and diagonalisable}.
On the other hand, for the resonant case $n=\lambda_1/\lambda_2 \in \N$ there are two possible outcomes:
\begin{itemize}
\item For $n \geq2$. The foliation $\fol$ is either linearisable (and diagonalisable), i.e., there are local
coordinates $(x,y)$ where $\fol$ is represented by the vector field
$$
x\frac{\partial}{\partial x}+ny\frac{\partial}{\partial y},
$$
or it is conjugate to the so-called {\it Poincar\'e-Dulac normal form}. In the latter case, in suitable local coordinates
$\fol$ is represented by the vector field
$$
x\frac{\partial}{\partial x}+(ny+x^n)\frac{\partial}{\partial y}.
$$
In particular, $\fol$ possesses a unique separatrix which is given in Poincar\'e-Dulac coordinates by
$\{x=0\}$ (in particular, this separatrix is necessarily smooth).

\item Assume now that $n=1$. In this case, the foliation $\fol$ is always linearisable but there are still two possibilites.
Namely, $\fol$ can be disgonalisable, i.e., conjugate to the foliation represented by the vector field
$$
x\frac{\partial}{\partial x}+y\frac{\partial}{\partial y} \, ,
$$
or non-diagonalisable. In the latter case, there are Poincar\'e-Dulac coordinates where $\fol$
is represented by the vector field
$$
x\frac{\partial}{\partial x}+(y+x)\frac{\partial}{\partial y} \, .
$$
Once again, when admitting the Poincar\'e-Dulac normal form, $\fol$ possesses a unique separatrix $\{x=0\}$
which happens to be smooth.
\end{itemize}

\begin{remark}\label{holonomy_separatrix-poincaredulac}
To complement the information on foliations admitting the Poincar\'e-Dulac normal form, we note that the local
holonomy map $h$ arising from the (unique) separatrix is such that its $n^{\rm th}$-iterate $h^n$ is tangent to
the identity while never {\it equal to the identity} as it can directly be checked.
\end{remark}

Armed with the above material on normal forms, we can go back to our approach to Theorem~\ref{a&b}.
Recall that remarks at the beginning of the section and Lemma~(\ref{no_saddlenodes}) imply that a Riccati
foliation exhibiting a hyperbolic or elliptic holonomy map around a simple fiber $C$ must have exactly two singularities
in $C$. Moreover, each of these singularities admit two eigenvalues different from zero.

\begin{lemma}
\label{Twosingularpoints}
Consider a Riccati foliation $\fol$ on the Hirzebruch surface $F_1$. Let $C$ be a simple invariant fiber and assume
that the local holonomy map of $\fol$ around $C$ is elliptic.
Denote by $p,q \in C$ the two singular points of $\fol$ in $C$ and let $\lambda_1^p,\lambda_2^p$ and $\lambda_1^q,\lambda_2^q$
be the corresponding (non-zero) eigenvalues (where $\lambda_1^p,\lambda_1^q$ are associated with the direction of
the invariant fiber $C$). Then these eigenvalues satisfy the relation
\begin{equation}
\frac{\lambda_2^p}{\lambda_1^p} + \frac{\lambda_2^q}{\lambda_1^q} = 0 \label{indexformula}
\end{equation}
and both quotients $\lambda_2^p/\lambda_1^p$, $\lambda_2^q/\lambda_1^q$ lie in $\R$.
Furthermore, the foliation $\fol$ is linearisable around both $p$ and $q$. Finally, at both singular points,
$\fol$ possesses a smooth separatrix transverse to the invariant fiber $C$.
\end{lemma}

\begin{proof}
Equation~(\ref{indexformula}) is a direct application of the Camacho-Sad index formula in \cite{camacho}.
Let us then begin by showing that the singular points $p$ and $q$ are not hyperbolic, i.e., that the
quotient of the eigenvalues is a real number. For this, assume for a contradiction
that the singular point $p$ is hyperbolic. Owing to the index formula~(\ref{indexformula}),
there follows that so is $q$. It follows from Poincar\'e theorem that $\fol$ is linearisable at both~$p$ and~$q$.
In particular, $\fol$ admits a smooth separatrix $S_p$ transverse to $C$ at $p$. Again $S_p$
corresponds to a fixed point of the holonomy map of $\fol$ around $C$. However, $p$ being hyperbolic,
there follows from an elementary calculation that this fixed point is hyperbolic, i.e., the absolute value
of its multiplier is different from~$1$. This is impossible since this holonomy map is elliptic.
The resulting contradiction then shows that the quotients $\lambda_2^p/\lambda_1^p$ and
$\lambda_2^q/\lambda_1^q$ lie in $\R$ and are different from zero.

From Formula~(\ref{indexformula}), we can assume without loss of generality that
$\lambda_2^p/\lambda_1^p >0$ and that $\lambda_2^q/\lambda_1^q <0$. Thus
the singular point $p$ belongs to the Poincar\'e domain while $q$ lies in
the Siegel domain. In particular,
either $\fol$ is linearisable around $p$ or it is conjugate to a Poincar\'e-Dulac
normal form.

\vspace{0.1cm}

\noindent {\it Claim}. $\fol$ cannot admit a Poincar\'e-Dulac normal form around~$p$.

\noindent {\it Proof of the Claim}. Assume that the claim is false. Thus $\fol$ possesses a unique
separatrix at~$p$ which must coincide with the fiber $C$. In particular
$\lambda_2^p/\lambda_1^p = n \in \N^{\ast}$. Meanwhile,
Formula~(\ref{indexformula}) ensures that at the singular point~$q \in C$ of $\fol$
the eigenvalues $\lambda_1^q$, $\lambda_2^q$ can be chosen as 
$\lambda_1^q = 1$ and $\lambda_2^q = -n$.
In particular, there exists a separatrix $S_q$ of $\fol$ at $q$ which is transverse to the invariant
fiber $C$. Again $S_q$ can naturally be identified with a fixed point of the holonomy map of
$\fol$ around the fiber $C$. Since this holonomy map is elliptic, it must be of finite order equal to~$n$.
Indeed, the linear part of this map at the fixed point represented by $S_q$ is determined by the eigenvalues
$\lambda_1^q = 1$ and $\lambda_2^q = -n$ and hence it is a rotation of order~$n$. Thus the $n^{\rm th}$-iterate
of the holonomy map in question has a fixed point with multiplier equal to~$1$. Since the original 
holonomy map is an elliptic element, we conclude that its $n^{\rm th}$-iterate must coincide with the identity.

Consider the local structure of $\fol$ around the Siegel singular point~$q$, the previously mentioned
theorem of Mattei-Moussu ensures that $\fol$ is linearisable at~$q$ since so is the local holonomy map arising
from $S_q$ (a finite order map). Given that the eigenvalues are $1$ and~$-n$, we see that the local holonomy
map of $\fol$ at~$q$ determined by the fiber $C$ (the other local separatrix of $\fol$ at~$q$) must coincide
with the identity. To derive a contradiction proving the claim, just notice that 
$C \setminus \{ p,q\}$ is a cylinder so that loops winding once around the singular
point~$p$ are homotopic to loops winding once around the singular point~$q$ (up to choosing suitable orientations).
Thus the local holonomy map around~$p$ arising from the separatrix of $\fol$ at~$p$ induced by $C$ is
the identity. In view of Remark~\ref{holonomy_separatrix-poincaredulac}, this contradicts the fact that
$\fol$ is locally conjugate to a Poincar\'e-Dulac vector field. The claim is proved.\qed

To complete the proof of the lemma, note first that the claim implies that $\fol$ is linearisable around~$p$.
In particular, the separatrix $S_p$ of $\fol$ at $p$ which is transverse to $C$ does exist. Moreover,
in terms of the separatrix induced by $C$, its associated (local) holonomy map is linearisable as well as
a germ in ${\rm Diff}\, (\C,0)$. Again using the fact that $C \setminus \{ p,q\}$ is a cylinder,
there follows that the (local) holonomy map arising from the separatrix induced by $C$ {\it at the singular
point~$q$}\, is linearisable as well and, hence, Mattei-Moussu's result ensures that $\fol$ is linearisable at~$q$.
The lemma follows at once.
\end{proof}

Recall that among discrete subgroups of ${\rm PSL}\, (2,\C )$, i.e., among Kleinian groups, there are some
groups that happen to leave invariant a (real analytic) circle $S^1$ in $\CP^1$.
These groups are therefore conjugate to discrete subgroups of ${\rm PSL}\, (2,\R ) \subset {\rm PSL}\, (2,\C )$
so that they can be viewed as discrete groups acting on the hyperbolic disc and hence are called
{\it Fuchsian groups}. The limit sets of Fuchsian groups are naturally contained in the invariant circle $S^1$.
In fact, up to identifying the initial group to a subgroup of ${\rm PSL}\, (2,\R )$, the invariant circle
becomes identified with the boundary of the unit disc.
Finally, recall that a Fuchsian group is said to be of {\it first kind} if its limit set coincides with all of the unit circle.

\begin{lemma}
\label{Limitsetaccumulation}
Let $\fol$ be a Riccati foliation on the surface $F_1$ all of whose invariant fibers $C_1, \ldots ,C_k$ are
simple. Assume that the following conditions hold:
\begin{itemize}
  \item The local holonomy map around each invariant fiber $C_i$ is an elliptic element of ${\rm PSL}\, (2,\C )$.
  \item The global holonomy group $\Gamma \subset {\rm PSL}\, (2,\C )$ of the foliation $\fol$ is a fuchsian/quasifuchsian
  group.
\end{itemize}
Then the intersection $C_i \cap \overline{\fol \, (\Lambda)}$ is reduced to the singular point $p_i \in C_i$ (that belongs to
the Poincar\'e domain) of $\fol$. Moreover, the eigenvalues of $\fol$ at $p_i$ belong, in fact,
to $\Q_+$ and the foliation is linearisable around the singular point $p_i$.
\end{lemma}

\begin{proof}
Let $C_i$ be a (simple) invariant fiber. According to Lemma~\ref{Twosingularpoints} both singular points $p_i$ and $q_i$
have real (non-zero) eigenvalues and the foliation is linearisable around each of them. With the preceding
notation, these eigenvalues verify $\lambda_2^{p_i}/\lambda^{p_i}_1 = - \lambda_2^{q_i}/\lambda^{q_i}_1 \in \R^{\ast}$.
Also, the separatrix of $\fol$
at $p_i$ (resp. $q_i$) transverse to $C_i$ will be denoted by $S_{p_i}$ (resp. $S_{q_i}$). Consider a generic fiber
$\Sigma \subset F_1$ near $C_i$ in the sense that on a fiberd neighbourhood of $C_i$ the separatrices
$S_{p_i},S_{q_i}$ intersect $\Sigma$ at unequivocally determined points. Then let these points be denoted 
by $x_{p_i},x_{q_i}$, respectively. As already mentioned, the points $x_{p_i}$, $x_{q_i}$ are fixed by the local holonomy map of
$\fol$ around $C_i$. A direct calculation using linearising coordinates for $\fol$ around $p_i$, $q_i$ shows that
the multipliers of the (elliptic) holonomy map in question at the fixed point $x_{p_i}$ (resp. $x_{q_i}$) are given
by $e^{2\pi i \lambda_2^{p_i}/\lambda^{p_i}_1}$ (resp. $e^{-2\pi i \lambda_2^{p_i}/\lambda^{p_i}_1}$). This elliptic
holonomy map must, however, be of finite order since it belongs to a discrete subgroup of ${\rm PSL}\, (2,\C )$.
It follows that $\lambda_2^{p_i}/\lambda^{p_i}_1 \in \Q_+^{\ast}$.

It only remains to prove that $C_i \cap \overline{\fol \, (\Lambda)} = \{ p_i \}$. The argument is similar
to the one employed in Section~\ref{ahlforscurrent_examples}. In the sequel, we take advantage of the fact
that $\fol$ is linearisable at both $p_i$ and $q_i$. First, we claim that
$C_i \cap \overline{\fol \, (\Lambda)} \subset \{p_i,q_i\}$. If the claim is false then $\fol(\Lambda)$ accumulates on
regular points of the invariant fiber $C_i$. Since the foliation $\fol$ has a (linearisable) saddle singular
point at $q_i$, the preceding implies that $\fol(\Lambda)$ accumulates on a regular points of the separatrix
$S_{q_i}$ as well. Hence the fixed point $x_{q_i} \in \Sigma$ belongs to $\overline{\fol \, (\Lambda)}$. However,
$\Sigma \cap \overline{\fol \, (\Lambda)}$ is nothing but the limit set of $\Gamma$ represented in the transverse
fiber $\Sigma$. This is impossible since, in a Fuchsian group, no fixed point of an elliptic
element can lie in the corresponding limit set. Thus we must have $C_i \cap \overline{\fol \, (\Lambda)} \subset \{p_i,q_i\}$.
However, since $\fol$ is linearisable around $q_i$, if $\fol(\Lambda)$ accumulates at $q_i$ then it must accumulate
at regular points of $S_{q_i}$ as well so that the argument above rules this possibility out as well. In conclusion, we
have proved that $C_i \cap \overline{\fol \, (\Lambda)} = \{ p_i \}$ and the lemma follows.
\end{proof}

We can now state Proposition~(\ref{LeviFlat}) which yields a criterion for a Riccati foliation to possess a
{\it singular real-analytic} Levi-flat where all leaves of the induced (Levi) foliation, when the global holonomy
group of the Riccati foliation $\fol$ is a Fuchsian group of ${\rm PSL}\, (2,\C )$ of first kind, are dense
(in the Levi-flat itself).

\begin{prop}
\label{LeviFlat}
Let $\fol$ be a Riccati foliation on the surface $F_1$ all of whose invariant fibers $C_1, \ldots ,C_k$ are
simple. Assume that the following conditions hold:
\begin{itemize}
  \item The local holonomy map around each invariant fiber $C_i$ is an elliptic element of ${\rm PSL}\, (2,\C )$.
  \item The global holonomy group $\Gamma \subset {\rm PSL}\, (2,\C )$ of the foliation $\fol$ is a Fuchsian group.
\end{itemize}
Then there exists a singular real analytic variety $\mathcal{L} (\fol)$ of (real) dimension~$3$ invariant by the foliation
$\fol$. Furthermore, the singular points of $\mathcal{L} (\fol)$ are all isolated and correspond to the
singular points of the foliation $\fol$ lying in the Poincar\'e domain. 
\end{prop}

\begin{proof}
Denote by $\fol(\Lambda)$ the saturated set of the limit set $\Lambda$ by the foliation $\fol$ and
$\overline{\fol \, (\Lambda)}$ the closure of $\fol(\Lambda)$. Clearly, $\fol(\Lambda)$ is {\it locally} a
real-analytic 3-dimensional manifold invariant by the foliation $\fol$ and, hence, satisfies the Levi condition
for flatness. To prove the proposition we need to describe the accumulation of the set $\fol(\Lambda)$ on each invariant fiber
$C_i$, $i=1,\cdots,k$. In this direction,
Lemma~(\ref{Limitsetaccumulation}) ensures that the intersection $C_i \cap \overline{\fol \, (\Lambda)}$ is reduced
to the singular point $p_i$ whose eigenvalues $\lambda_1^{p_i}$ and $\lambda_2^{p_i}$ satisfy
$\lambda_2^{p_i}/\lambda_1^{p_i} \in \Q_+$.
The proof of the proposition now follows from Lemma~\ref{addedlemmaanalytic}
below stating that $\overline{\fol \, (\Lambda)}$ still is an analytic variety
around $p_i$.
\end{proof}

\begin{lemma}
	\label{addedlemmaanalytic}
	With the notation of Proposition~\ref{LeviFlat}, the set $\overline{\fol \, (\Lambda)}$ is locally real
	analytic around $p_i$.
\end{lemma}

The proof of Lemma~\ref{addedlemmaanalytic} will be deferred to Section~\ref{buldingharmoniccurrents-I}
since it employs some blowing up procedure that, incidentally, will also be instrumental in constructing 
the harmonic current $T$ mentioned in items~(c) and~(d) of Theorem~A.

As a by-product of the previous discussion, we also obtain Proposition~(\ref{LeviFlatII}) below which
is valid for Riccati foliations whose global holonomy groups are quasifuchsian groups and which
dispenses with Lemma~\ref{addedlemmaanalytic}.

\begin{prop}
\label{LeviFlatII}
Let $\fol$ be a Riccati foliation on $F_1$ all of whose invariant fibers $C_1, \ldots ,C_k$ are
simple. Assume that the following conditions hold:
\begin{itemize}
  \item The local holonomy map $M_i$ around each invariant fiber $C_i$ is an elliptic element of ${\rm PSL}\, (2,\C )$.
  \item The global holonomy group $\Gamma \subset {\rm PSL}\, (2,\C )$ of $\fol$ is a quasifuchsian group.
\end{itemize}
Then there exists a closed set $\mathcal{L} (\fol)$ invariant by the foliation $\fol$ which is a singular topological
manifold of (topological) dimension~$3$ and Hausdorff dimension strictly greater than~$3$.
Furthermore, the singular points of $\mathcal{L} (\fol)$ are all isolated and correspond
to the singular points of the foliation $\fol$ lying in the Poincar\'e domain.
\end{prop}

\begin{proof}
The construction of $\mathcal{L} (\fol)$ is exactly as in the previously considered Fuchsian case
(see Proposition~\ref{LeviFlat}). The additional difficulty of proving that $\mathcal{L} (\fol)$ is real
analytic plays no role in the present case. Meanwhile, the only issue that requires explanation is the claim
that the Hausdorff dimension of $\mathcal{L} (\fol)$ is strictly greater than~$3$. This, however,
follows from Bowen's theorem \cite{bowen} asserting that the Hausdorff dimension of a quasicircle is
strictly greater than~$1$. Thus, $\mathcal{L} (\fol)$ can be pictured as a lamination by Riemann surfaces
which transversely has Hausdorff dimension strictly greater than~$1$ so that the Hausdorff dimension of
the laminated space $\mathcal{L} (\fol)$ is strictly greater than~$3$.
\end{proof}

Recalling that the direct image of a real analytic set by an analytic map may fail to be real analytic,
the last ingredient needed in the proof of Theorem~\ref{a&b} is Lemma~\ref{RationalCurve} below. 

\begin{lemma}
\label{RationalCurve}
Let $\fol$ and $\overline{\fol \, (\Lambda)}$ be as in Proposition~\ref{LeviFlatII}. Denote by $\mathcal{C}$
the $(-1)$-rational curve contained in the Hirzebruch surface $F_1$. Then we have
$$
\overline{\fol \, (\Lambda)} \cap \mathcal{C} = \phi \, .
$$ 
\end{lemma}

\begin{proof}
For each $i \in \{1, \ldots ,k\}$, let $b_i$ be the (unique) intersection point between the invariant fiber $C_i$
and the rational curve $\mathcal{C}$ and set $\mathcal{C}_0=\mathcal{C}\setminus \{b_1, \ldots ,b_k\}$. Let
$\Sigma$ be a non-invariant fiber of $F_1$ (hence transverse to $\fol$) and denote by $q$ the intersection point
between $\Sigma$ and $\mathcal{C}_0$. Consider then the holonomy representation $\rho : \pi_1(\mathcal{C}_0) \rightarrow \psl$
and let $\Gamma=\rho(\pi_1(\mathcal{C}_0))$ be the global holonomy group of the foliation $\fol$ acting on $\Sigma$.
Let $\overline{\mathcal{C}}_0$ be the covering space of $\mathcal{C}_0$ such that
$\pi_1(\overline{\mathcal{C}}_0)={\rm ker}\,(\rho)$ and denote by $\calp : \overline{\mathcal{C}}_0 \rightarrow \mathcal{C}_0$
the corresponding covering projection.

Fix a base point $\overline{p}$ in $\overline{\mathcal{C}}_0$ and let $p=\calp(\overline{p}) \in \mathcal{C}_0$.
Let $\gamma:[0,1] \rightarrow \mathcal{C}_0$ be a path joining $p$ to $q$ (i.e., $\gamma(0)=p$ and $\gamma(1)=q$), then
$\gamma$ can be lifted to the leaf $L_p$ of the foliation $\fol$ through the point $p$ thanks to the fact that
the foliation $\fol$ is a Riccati foliation. The lift $\overline{\gamma}:[0,1] \rightarrow L_p$ joins $p$
to a point, denoted by $x = \mathcal{D} (\overline{p})$, in the intersection $L_p \cap \Sigma$.

The correspondence between the point $\overline{p} \in \overline{\mathcal{C}}_0$ and the
point $x = \mathcal{D} (\overline{p}) \in \Sigma$ can naturally be extended over paths in $\overline{\mathcal{C}}_0$.
To do so, consider a path $\overline{\sigma} : [0,1] \rightarrow
\overline{\mathcal{C}}_0$ such that $\overline{\sigma} (0) = \overline{p}$. To define
$\mathcal{D} (\overline{\sigma} (1))$, denote by $\sigma : [0,1] \rightarrow \mathcal{C}_0$ the projection of $\sigma$
onto $\mathcal{C}_0$, $\mathcal{D} (\overline{\sigma} (1))$ can then be defined to be the terminal point of the lift of
the path $\sigma^{-1} \ast \gamma$ in $L_{\sigma(1)}$. Let $\overline{\sigma}_1 : [0,1] \rightarrow \overline{\mathcal{C}}_0$
be a deformation of $\overline{\sigma}$ with fixed endpoints, then it is easy to check that $\mathcal{D} (\overline{\sigma} (1))=
\mathcal{D} (\overline{\sigma}_1 (1))$. Thus, $\mathcal{D}$, in fact, only depends on the homotopic class
of $\overline{\sigma}$.

We need then to show that $\mathcal{D}$ is well defined. Let $\overline{c}$ be a closed loop in $\overline{\mathcal{C}}_0$
with a base point at $\overline{p}$. Then $c=\calp(\overline{c})$ is a closed loop in $\mathcal{C}_0$ with a base point at $p$. The
path $c^{-1}$ can be lifted to $L_p$. Thus, the lift $c_0:[0,1] \rightarrow L_p$ joins $p$ to a point, denoted by $p_0$, in the
intersection $L_p \cap \Sigma$. Meanwhile, $p_0=\rho(c^{-1}) \cdot p= {\rm id} \cdot p=p$ since
$c^{-1} \in \pi_1(\overline{\mathcal{C}}_0)={\rm ker}\,(\rho)$. It follows that the terminal point of the lift of the path
$c^{-1} \ast \gamma$ in the leaf $L_p$ is, again, the point $x$ itself and, hence, the function $\mathcal{D}$ is well defined.

Let $G$ be the group of deck transformations acting on $\overline{\mathcal{C}}_0$. As an abstract group $G$ is given as the
quotient $$\pi_1(\mathcal{C}_0) / \pi_1(\overline{\mathcal{C}}_0)=\pi_1(\mathcal{C}_0) / {\rm ker}\,(\rho).$$ On the other hand,
the holonomy representation yields $$\pi_1(\mathcal{C}_0) / {\rm ker}\,(\rho)\simeq \rho(\pi_1(\mathcal{C}_0))=\Gamma,$$ where
$\Gamma$ is the global holonomy group of the foliation $\fol$. Thus, these two groups $G$ and $\Gamma$ are isomorphic. In other
words, $G$ and $\Gamma$, viewed as transformation groups of $\overline{\mathcal{C}}_0$ and $\Sigma$ respectively, can be interpreted
as two different actions of a same group, namely $\pi_1(\mathcal{C}_0) / {\rm ker}\,(\rho)$. Let $g$ be an element in $G$ and $\alpha$
the corresponding element in $\Gamma$. It follows from the construction above that
$$
\rho(\alpha) \cdot \mathcal{D}(\overline{r})=\mathcal{D}(g \cdot \overline{r}),
$$
for every $\overline{r} \in \overline{\mathcal{C}}_0$. In other words, the function $\mathcal{D}$ is {\it equivariant} with respect
to the actions of the groups $G$ and $\Gamma$.

Let $U = \mathcal{D}(\overline{\mathcal{C}}_0) \subset \Sigma$. Then the set $U$ is open (since the function $\mathcal{D}$ is open
by construction) and invariant under the action of the group $\Gamma$. There follows that $U / \Gamma \simeq \overline{\mathcal{C}}_0 / G$
and, since $\overline{\mathcal{C}}_0 / G$ is nothing but $\mathcal{C}\setminus \{b_1, \ldots ,b_k\}$, $U / \Gamma$ is a Hausdorff manifold.
Hence, the group $\Gamma$ acts properly discontinuously on $U$ which, in turn, implies that $U$ is entirely contained in the domain
of discontinuity of the quasifuchsian group $\Gamma$. Thus, $U \cap \Lambda = \phi$ where $\Lambda$ is the limit set of the
group $\Gamma$. The lemma then follows since $\fol(\mathcal{C}_0) \cap \Sigma =U$ and $\overline{\fol \, (\Lambda)} \cap \Sigma =\Lambda$.
\end{proof}

\begin{proof}[Proof of Theorem~\ref{a&b}]
The statement follows at once from the combination of Propositions~\ref{LeviFlat} and~\ref{LeviFlatII}
with Lemma~\ref{RationalCurve}. In fact, in the case of Fuchsian groups, the Levi-flat $\mathcal{L} \subset \CP^2$
is nothing but the image of the (Levi-flat) $\mathcal{L} (\fol) \subset F_1$
whose existence is ensured by Proposition~\ref{LeviFlat} through the blow down mapping from $F_1$ to $\CP^2$.
To show that the resulting set $\mathcal{L}$ satisfy the conditions of Theorem~\ref{a&b} just observe
that $\mathcal{L} (\fol)$ is contained in a compact part of $F_1 \setminus \mathcal{C}$ thanks to
Lemma~\ref{RationalCurve}, where $\mathcal{C}$ stands for the $(-1)$-rational curve in $F_1$. It follows that the
mentioned blow down mapping is actually a holomorphic diffeomorphism on a neighborhood
of $\mathcal{L} (\fol)$ so that the theorem follows. The same argument applies to the case of quasifuchsien groups
as discussed in Proposition~\ref{LeviFlatII}. Theorem~\ref{a&b} is proved.
\end{proof}

\section{Foliated harmonic currents}\label{buldingharmoniccurrents-I}
Having essentially proved the first two statements of Theorem~A in section~(\ref{provingTheoremA_Part_I}),
the present section is devoted to establishing statements~(c), (d), and~(e) in the theorem in question.
In other words, we will describe all foliated
harmonic currents for Riccati foliations on $\CP^2$ satisfying the conditions of Theorem~A.

To avoid useless repetitions, in the course of this section we shall consider that a Fuchsian group is a special
case of a quasifuchsian one. In other words, when referring to quasifuchsian groups in the sequel, the possibility
that the group is question is actually Fuchsian {\it is not}\, ruled out.
Along similar lines, we will denote by $\mathcal{L} (\fol) \subset F_1$ the Levi-flat arising
from the Riccati foliation considered in Propositions~\ref{LeviFlat} and~\ref{LeviFlatII} regardless of whether
$\mathcal{L} (\fol)$ is real-analytic or transversely fractal. Of course, every quasifuchsien/fuchsian group considered
in the sequel is assumed to be non-elementary. The notation used in the proofs
of Propositions~\ref{LeviFlat} and~\ref{LeviFlatII} will be resumed in the sequel.

Consider an invariant fiber $C_i$ of $\fol$ along with the singular points $p_i, q_i \in C_i$, where $p_i$ lies
in the Poincar\'e domain and $q_i$ in the Siegel domain. Denote by $m_i \geq 2$ the order of the (elliptic) holonomy
map of $\fol$ arising from winding around $C_i$. As previously seen, $\fol$ is linearisable around $p_i$.
Thus, there are local coordinates $(u_i,v_i)$ where $\fol$ is locally conjugate to the foliation associated with
the vector field
$$
m_i u_i \frac{\partial}{\partial u_i} + n_i v_i \frac{\partial}{\partial v_i} \, ,
$$
with $p_i \simeq (0,0)$ and $\{ u_i =0\} \subset C_i$. Moreover, $m_i$ and $n_i$ are strictly positive integers
which can be assumed to satisfy $1 \leq n_i < m_i$ without loss of generality. In particular, the separatrix
given by $\{ v_i =0\}$ is distinguished as the unique (local, smooth) separatrix that is transverse to
the fiber $C_i$. Apart from the separatrix induced by $C_i$, $\{ v_i =0\}$ is the unique separatrix
carrying non-trivial holonomy which is conjugate to the rotation of angle $2\pi n_i/m_i$. This separatrix
is denoted by $S_p$.

In the $(u_i,v_i)$-coordinates, $\fol$ admits the function $(u_i,v_i) \mapsto u_i^{n_i} v_i^{-m_i}$ as a
meromorphic first integral so that the leaves are locally algebraic of the form $u_i^{n_i} = c v_i^{m_i}$
for a suitable constant $c \in \C$. Yet a blow up procedure allows eliminate the indeterminacy point of the
mentioned first integral so as to describe the behavior of these
leaves in a way better adapted to our needs. In fact, it is useful carry out this procedure relying on the
following elementary observation: the blow up of the foliation in question yields a foliation leaving
invariant the resulting exceptional divisor and possessing exactly $2$ singular points. Furthermore, both
singular points have integer eigenvalues and one of them lies in the Poincar\'e domain while the other
lies in the Siegel domain. The procedure then amounts to carrying on blowing up the singularity lying in the
Poincar\'e domain until we reach a singular point whose eigenvalues are $1$ and $1$. Then one last
blow up at this singular point will lead to a foliation that is transverse to the corresponding component of
the exceptional divisor. In other words,
after finitely many blow ups - all centered at singular points lying in the Poincar\'e domain - we obtain
picture:
\begin{itemize}
	\item[(1)] An exceptional divisor $\mathcal{E}$ consisting of a string of rational curves $D_1, \ldots , D_l$
	with transverse intersections.
	
	\item[(2)] A foliation $\tilf$ that possesses only singularities lying in the Siegel domain.
	
	\item[(3)] A component $D_{J}$ of  $\mathcal{E}$ which is transverse to $\tilf$ (in particular in
	a neighborhood of $D_J$).
	
	\item[(4)] The remaining components $D_1, \ldots , D_{J-1}, D_{J+1} ,\ldots ,D_l$ are all invariant by
	$\tilf$.
	
	\item[(5)] $D_1$ intersects (the transform of) $C_i$ while (the transform of) $S_p$ determines a singular
	point of $\tilf$ lying in $D_l$.
\end{itemize}

We are now ready to prove Lemma~\ref{addedlemmaanalytic} and then proceed to the construction of the
current $T$ supported on $\mathcal{L} (\fol)$.

\begin{proof}[Proof of Lemma~\ref{addedlemmaanalytic}]
We consider a fibered neighborhood $U = \D \times C_i$ of $C_i$ where $\D \subset \C$. Fix a
fiber $\Sigma \subset U$ and let $\sigma : \Sigma \rightarrow \Sigma$ denote the (elliptic) holonomy
map arising from winding around $C_i$. We choose projective coordinates on $\Sigma$ such that
$\mathcal{L} (\fol) \cap \Sigma = \R \cup \{ \infty\}$. To prove the lemma, it suffices to construct a
meromorphic first integral $F$ for $\fol$ on $U$ such that $\mathcal{L} (\fol) \cap U$ coincides
with $F^{-1} (\R \cup \{ \infty\})$. Indeed, if $F$ is such an integral, then $F = \overline{F}$ yields an analytic
equation defining $\mathcal{L} (\fol) \cap U$.

Denote by $\xi : \CP^1 \rightarrow \CP^1$ is an elliptic automorphism of order~$m_i$ leaving $\R \cup \{ \infty\}$
invariant. Consider then a non-constant holomorphic map $f : \CP^1 \rightarrow \CP^1$ which is
invariant under $\xi$. For example, if we change coordinates such that $\xi$ becomes a rotation around
the origin and where $\R \cup \{ \infty\}$ becomes the unit circle, then we can choose $f(z) = z^n$. Moving back to the initial
coordinates then yields the desired invariant function.

Denote by $P_i, Q_i$ the two fixed points of $\xi$. Recall that $S_p$ (resp. $S_q$) is the separatrix issued from the singular
point $p_i$ (resp. $q_i$) which is transverse to $C_i$. As seen in Section~\ref{provingTheoremA_Part_I},
$S_p$ meets $\Sigma$ at the fixed point $P_i$. Similarly $S_q$ meets $\Sigma$ at $Q_i$.

Next apply to the singular point $p_i$ the blow up procedure described above.
We then denote by $\widetilde{U}$ the transform of $U$. Also we denote by $\mathcal{E}_1$ the divisor consisting
of the string of rational curves going from $C_i$ to $D_{J-1}$. Similarly, $\mathcal{E}_2$ is the string of
rational curves $D_{J+1}, \ldots , D_l$.

To complete the proof of the lemma, we will construct a holomorphic mapping
$\widetilde{F}: \widetilde{U} \rightarrow \CP^1$ which is constant over the leaves of $\tilf$ (the transform
of $\fol$). Clearly, such a map induces the desired meromorphic first integral $F$ on $U$ (Levi extension). In turn, to construct the
mapping $\widetilde{F} : \widetilde{U} \rightarrow \CP^1$ constant over the leaves of $\tilf$,
we first set $\widetilde{F} (z) = f(Q_i)$ if $z \in \mathcal{E}_1$ and
$\widetilde{F} (z) = f(P_i)$ if $z \in \mathcal{E}_2$. Now, if $z \in \widetilde{U} \setminus (\mathcal{E}_1 \cup \mathcal{E}_2)$,
then we consider the leaf $L_z$ of $\tilf \simeq \fol$ through $z$ and set $\widetilde{F}(z) = f( L_z \cap \Sigma)$.
It is immediate that $\widetilde{F}$ is well defined since all the value of $f$ at every intersection
point $L_z \cap \Sigma$ is the same ($f$ is invariant under $\xi$). Furthermore, since $\tilf$ is regular
away from $\mathcal{E}_1 \cup \mathcal{E}_2$, it is clear that $\widetilde{F}$ is holomorphic
on $\widetilde{U} \setminus (\mathcal{E}_1 \cup \mathcal{E}_2)$. Finally, to show that $\widetilde{F}$
is holomorphic on all of $\widetilde{U}$, it suffices to show that this mapping is continuous 
at $\mathcal{E}_1$ and at $\mathcal{E}_2$.

The continuity of $\widetilde{F}$ at points in $\mathcal{E}_1 \cup \mathcal{E}_2$, however, follows
from the structure of (linear) Siegel singular points detailed in Section~\ref{provingTheoremA_Part_I}.
In fact, let $\{ z_i\} \subset \widetilde{U} \setminus (\mathcal{E}_1 \cup \mathcal{E}_2)$ be
a sequence of points converging towards a point in, say, $\mathcal{E_1}$. Since all singular points
of $\tilf$ lie in the Siegel domain, it follows that the corresponding leaves $L_{z_i}$ accumulate on the
separatrix $S_q$. Thus, the intersection points $L_{z_i} \cap \Sigma$ cluster around $Q_i$ so that
$\widetilde{F} (z_i)$ converges to $f(Q_i)$ hence establishing the continuity of $\widetilde{F}$ at
points in $\mathcal{E}_1$. The analogous argument shows that $\widetilde{F}$ is also continuous at
points in $\mathcal{E}_2$ so that Riemann extension implies that $\widetilde{F}$ is holomorphic on all
of $\widetilde{U}$. The proof of the lemma is completed.
\end{proof}

The remainder of the section is devoted to constructing the harmonic current $T$ supported on 
$\mathcal{L} (\fol)$ in order to complete the proof of Theorem~A. Let us then go back to the
Riccati foliation $\fol$ on $F_1$ as in Proposition~\ref{LeviFlatII}.
The global holonomy group $\Gamma$ of $\fol$ is a quasifuchsian group of ${\rm PSL}\, (2,\C )$ (the Fuchsian
case viewed as a particular one). Furthermore, since
$\Gamma$ is assumed to be of first kind so that its limit set is all of the invariant Jordan curve.
The Levi-flat $\mathcal{L} (\fol) \subset F_1$ obtained as the closure of the saturated of this Jordan
curve by the foliation $\fol$
is a singular topological manifold of dimension~$3$ whose singular points coincide with the singular points
of $\fol$ belonging to the Poincar\'e domain. Clearly, every leaf of $\fol$ contained in $\mathcal{L} (\fol)$
is dense in $\mathcal{L} (\fol)$ since $\Gamma$ is of first kind.

Now for each $i \in \{1, \ldots ,k\}$, we apply the blowing up procedure used in the proof of
Lemma~\ref{addedlemmaanalytic} to each of the singular points $p_i \in C_i$ lying in the Poincar\'e domain.
Denote by $N$ the resulting surface and by $\tilf$ the transform of $\fol$. Note that the corresponding
projection $\Pi : N \rightarrow F_1$
is a holomorphic diffeomorphism from the complement of the total exceptional divisor in $N$ to
$F_1 \setminus \{ p_1, \ldots, p_k\}$. Then set $\mathcal{R}_i = C_i \cup \Pi^{-1} (p_i)$ where $C_i$ is
identified with its transform. Clearly, $\mathcal{R}_i$ is a string of rational curves containing a unique
curve $D_{J_i}$ which is transverse to $\tilf$ while all the other components are invariant by $\tilf$.
In the sequel, let $\widetilde{\mathcal{L} (\fol)}$ be the transform of the Levi-flat $\mathcal{L} (\fol)$.

\begin{lemma}
\label{smoothLeviflat}
$\widetilde{\mathcal{L} (\fol)}$ is a regular topological manifold. It intersects the divisor
$\mathcal{R}_i$ on a Jordan curve contained in the curve $D_{J_i}$.
\end{lemma}

\begin{proof}
With the notation of Lemma~\ref{addedlemmaanalytic}, $\mathcal{R}_i = C_i \cup \mathcal{E}_i$ and each
of the connected components of $\mathcal{R}_i \setminus D_{J_i}$ are contained in saturated sets of
arbitrarily small neighborhoods of the local monodromy. Therefore, they remain away from the Jordan curve
arising as limit set of $\Gamma$. Hence, $\widetilde{\mathcal{L} (\fol)} \cap \mathcal{R}_i$ is contained
in a compact part of $D_{J_i}$ minus the intersection points with $D_{J_i -1}$ and $D_{J_i +1}$. In turn,
the foliation $\tilf$ is regular at these intersection points and the lemma follows.
\end{proof}

We can easily compare the leaves of $\tilf$ in $\widetilde{\mathcal{L} (\fol)}$ with the leaves of
$\fol$ in $\mathcal{L} (\fol)$. In fact, it might be more accurate to talk about the {\it filled leaves}
of the foliation $\fol$ which are defined as follows. First, a {\it leaf}\, of $\fol$ in $\mathcal{L} (\fol)$ is
nothing but a leaf of the non-singular foliation obtained by restricting $\fol$ to $F_1 \setminus \{ p_1, \ldots, p_k\}$.
If $L$ is one such leaf, then the corresponding {\it filled leaf}\, $\overline{L}$ is
defined by adding the singular point $p_i$ to every local branch of the leaf $L$ passing through $p_i$,
where each branch has the local form $u_i^{n_i} = c v_i^{m_i}$ for suitable $c \in \C$ and positive integers $m_i,n_i$.
Now the comparison between leaves of $\tilf$ in $\widetilde{\mathcal{L} (\fol)}$ and leaves of
$\fol$ in $\mathcal{L} (\fol)$ is made accurate by the following lemma.

\begin{lemma}
	\label{projecting_leaves-filledversion}
	If $\widetilde{L}$ is a leaf of the foliation $\tilf$ contained in $\widetilde{\mathcal{L} (\fol)}$
	then the restriction of $\Pi$ to $\widetilde{L}$ is a diffeomorphism between $\widetilde{L}$ and some filled leaf
	$\overline{L}$ of the foliation $\fol$.
\end{lemma}

\begin{proof}
	In view of the preceding, the intersection of a leaf $\widetilde{L}$ with the exceptional divisor of $\Pi$
	is contained in the union of the curves $D_{J_i}$. These are compact (rational) curves transverse to
	$\tilf$. Therefore, the intersection of $\widetilde{L}$ and the exceptional divisor of $\Pi$
	is uniformly transverse (i.e., transverse with angle bounded from below by a strictly positive constant).
	The lemma follows immediately.	
\end{proof}

Next, we have:

\begin{lemma}
	\label{filledleavesarehyperbolic}
	The filled leaves $\overline{L}$ of $\fol$ in $\mathcal{L} (\fol)$ are hyperbolic Riemann surfaces.
\end{lemma}

\begin{proof}
To show that $\overline{L}$ is a hyperbolic Riemann surface it suffices to check
that its volume grows exponentially. For regular foliations everywhere transverse to a fibration, the growth type of
leaves is determined by the growth of the global holonomy group of the foliation which is exponential since
it is a quasifuchsian group. A minor difficulty here arises from the fact that $\overline{L}$ fails
to be transverse to $C_i$ at the point $p_i$.

The indicated issue is, however, settled by Lemma~\ref{projecting_leaves-filledversion}. Around each point $p_i$, we
can place a small ball $B_i$ such that the away from $B_i$ the leaf is transverse to the fibers and hence
has its volume growth comparable with the exponential growth of the holonomy group $\Gamma$. Furthermore,
owing to Lemma~\ref{projecting_leaves-filledversion}, the intersection of each local branch of $\overline{L}$
with $B_i$ is a disc of small area and small (comparable) diameter. Combining these informations, it becomes
clear that the volume of $\overline{L}$ grows exponentially so that the statement follows.
\end{proof}

Before stating the next proposition, it is convenient to recall the notion of self-intersection for
harmonic currents introduced in \cite{fornaessSibony}. Consider a positive harmonic current $T$ on a compact
K\"ahler surface. By using Hodge theory, Fornaess and Sibony managed to define the {\it self-intersection of $T$}
by showing that the integral
$$
\int T \wedge T
$$
is well defined for positive harmonic currents. Moreover, this integral coincides with the usual formulation
in the case where $T$ is smooth. In particular, in the case of $\CP^2$, they mentioned the problem of computing the
quantity
$$
\inf \left\{ \int T \wedge T \, ; \; \;   T \geq 0 \; \; {\rm and} \; \; i\partial \overline{\partial} T =0  \right\} \, .
$$
Proving that the infimum in question is strictly positive would have major implications in several well-known
conjectures about Riemann surface laminations in $\CP^2$. A by-product of their second paper \cite{fornaesssibony-2},
however, is that this infimum equals
zero. Yet, it seem that no example of positive harmonic current with zero self-intersection and having,
say, support with empty interior was previously identified so that applications to the quoted problems could still
be envisaged. In this regard, the contribution of Theorem~A is summarised by its item~(e).

The following proposition is more naturally stated in the language of {\it laminations}, see for example
\cite{fornaesssibony-survey}.

\begin{prop}\label{Currents_upstairs}
	The Riemann surface lamination defined by the restriction of $\tilf$ to $\widetilde{\mathcal{L} (\fol)}$
	possesses a unique positive foliated harmonic current $\widetilde{T}$. The current $\widetilde{T}$
	is not closed and has zero self-intersection. Finally, the support of the current $\widetilde{T}$ coincides
	with all of $\widetilde{\mathcal{L} (\fol)}$.
\end{prop}

\begin{proof}
As shown by the preceding lemmas, the restriction of $\tilf$ to $\widetilde{\mathcal{L} (\fol)}$ is a regular
lamination by Riemann surfaces all of whose leaves are hyperbolic. We claim that this lamination admits no
positive foliated closed current. To check the claim, just note that one such closed current would yield
a transverse invariant measure for the lamination in question. This measure would project through $\Pi$ into
a finite measure invariant by the holonomy group $\Gamma$ on its limit set, which immediately gives rise to a contradiction since
$\Gamma$ is non-elementary. This contradiction proves the claim.

The existence of a unique positive foliated harmonic current $\widetilde{T}$ on $\widetilde{\mathcal{L} (\fol)}$
now follows from the main result in \cite{sibony_uniqueness}. This same theorem also shows that the self-intersection
of $\widetilde{T}$ equals zero. Finally, since the action of $\Gamma$ on its limit set has all orbits dense,
it follows that all the leaves of the lamination induced on $\widetilde{\mathcal{L} (\fol)}$ by $\tilf$ are
dense in $\widetilde{\mathcal{L} (\fol)}$. Hence, the support of $\widetilde{T}$ must coincide with
all of $\widetilde{\mathcal{L} (\fol)}$.
\end{proof}

Recall that the projection $\Pi : N \rightarrow F_1$ is holomorphic and globally defined on $N$.
The current $\widetilde{T}$ can then be pushed forward by $\Pi$ to yield a positive foliated harmonic
current $T=\Pi_{\ast}\widetilde{T}$, for $\fol$ supported on all of $\mathcal{L} (\fol)$ which,
in addition, has null self-intersection. The last ingredient needed in the proof of
Theorem~A is the following proposition:

\begin{prop}\label{Pullingback_Currents}
	Every $(1,1)$-foliated harmonic current $T$ supported in $\mathcal{L} (\fol)$ can be pulled-back
	by $\Pi$ to yield a foliated harmonic current for $\tilf$ supported on $\widetilde{\mathcal{L} (\fol)}$.
\end{prop}

\begin{proof}
Let $T$ be as in the statement and consider a $(1,1)$-differential form $\omega$ on $N$.
To define the pull-back $\widetilde{T}= \Pi^{\ast} T$ it suffices to define a push-forward $\Pi_{\ast} \omega$ for
$\omega$ so that the coupling $\langle T , \Pi_{\ast} \omega \rangle$ makes sense. In more accurate terms,
$\Pi$ is a diffeomorphism between a Zariski-open subset of $N$ and $F_1 \setminus \{p_1, \ldots ,p_k \}$. Thus,
$\Pi_{\ast} \omega$ is naturally a $(1,1)$-differential form {\it defined on $F_1 \setminus \{p_1, \ldots ,p_k \}$}.
In principle, however, the form $\Pi_{\ast} \omega$ may behave wildly near the points $p_1, \ldots ,p_k$. The
proposition will follow from checking that the coupling $\langle T , \Pi_{\ast} \omega \rangle$ is, nonetheless, well defined
and yields a continuous functional on the space of $(1,1)$-forms on $N$.

It is enough to work on a neighborhood of a point $p_i$. As previously seen, there are local coordinates
$(u_i,v_i)$ around $p_i \simeq (0,0)$ where $\fol$ is locally given by the vector field
$$
m_i u_i \frac{\partial}{\partial u_i} + n_i v_i \frac{\partial}{\partial v_i}
$$
where $m_i$ and $n_i$ are strictly positive integers and such that $\{ u_i =0\} \subset C_i$. We then apply the sequence of blow ups
described at the beginning of the section to the point $p_i$ so as to remove the indetermination of the first integral
$u_i^{n_i} v_i^{-m_i}$. Resuming the notation used in the proof of Lemma~\ref{addedlemmaanalytic}, we recall 
that the transform of $\mathcal{L} (\fol)$ intersects transversely the corresponding exceptional divisor at a unique component $D_{J_i}$.
In fact, $D_{J_i}$ is a rational curve of self-intersection $-1$ since it arises from the last (one-point) blow up performed
in our blow up procedure. There are, therefore, affine (blow up) coordinates $(t_i, s_i)$, $\{ t_i =0\} \subset D_{J_i}$
on a neighborhood of $D_{J_i}$ satisfying the following conditions:
\begin{itemize}
	\item[(1)] The foliation $\tilf$ is locally given by the vector field $\partial /\partial t_i$.

	\item[(2)] The blow down map $\Pi : N \rightarrow F_1$ is locally given by $\Pi (t_i s_i) = (t_i^{m_i} s_i ,t_i^{n_i}) = (u_i,v_i)$.
\end{itemize}
In particular, the local inverse of $\Pi$ is determined in ramified coordinates by $t_i = \sqrt[n_i]{v_i}$ and $s_i = u_i /\sqrt[n_i]{v_i^{m_i}}$.

Now let $\omega$ be a $(1,1)$-differential form on the surface $N$ whose support intersects $D_{J_i}$. Set
$$
\omega=a_{11}dt_i \wedge d\overline{t}_i+a_{12}dt_i \wedge d\overline{s}_i+
a_{21}d\overline{t}_i \wedge ds_i+a_{22}ds_i \wedge d\overline{s}_i
$$
in the local coordinates $(t_i ,s_i)$. Recall that $\Pi_{\ast} \omega$ is defined away from the exceptional divisor of $\Pi$.
Over the open set $F_1 \setminus \{p_1, \ldots ,p_k \}$, we have
\[
\langle T, \Pi_{\ast} (a_{12}dt_i \wedge d\overline{s}_i) \rangle= \langle T, \Pi_{\ast}(a_{21}d\overline{t}_i \wedge ds_i) \rangle
= \langle T, \Pi_{\ast}(a_{22}ds_i \wedge d\overline{s}_i) \rangle = 0 \, ,
\]
since $T$ is foliated and the corresponding push-forwards vanish identically over the tangent space of the foliation $\fol$.
Thus, it only remains to show that the coupling
\[
\langle T,\Pi_{\ast}(a_{11}dt_i \wedge d\overline{t}_i) \rangle
\] 
is well defined on $F_1 \setminus \{p_1, \ldots ,p_k \}$. To do so, recall first that $T$ is harmonic so that it is represented
in flow-boxes by
$$
S=\int_{\Sigma_{\alpha}} h_{\alpha} [\Delta_{\alpha}]d\mu(\alpha),
$$
where $\mu(\alpha)$ is a positive Borel measure on the transverse sections $\Sigma_{\alpha}$ and where
$h_{\alpha}$ stand for strictly positive
harmonic functions, uniformly bounded above and below by strictly positive constants. Clearly, the functions $h_{\alpha}$ are Borel measurable
with respect to $\alpha$. Thus, away from $p_i$, the coupling $\langle T,\Pi_{\ast}(a_{11}dt_i \wedge d\overline{t}_i) \rangle$
$$
\int_{\Sigma_{\alpha}} \left( \int_{L_{\alpha}} h_{\alpha} \Pi_{\ast}(a_{11}dt_i \wedge d\overline{t}_i) \right)d\mu(\alpha) \, ,
$$
where $L_{\alpha}$ is the corresponding leaf of $\fol$. Thus, it suffices to show that the integral
$$
\int_{L_{\alpha}} \Pi_{\ast}(a_{11}dt_i \wedge d\overline{t}_i)
$$
is bounded (uniformly on $\alpha$). To check that this is the case, note first that
$$
\Pi_{\ast}(a_{11}dt_i \wedge d\overline{t}_i) = a_{11} (v_i^{1/n_i}, u_i v_i^{-m_i/n_i}) dt_i \wedge d\overline{t}_i \, .
$$
In particular, the coefficient $a_{11} (v_i^{1/n_i}, u_i v_i^{-m_i/n_i})$ is identified with a continuous function on a
neighborhood of $p_i$ since $u_i v_i^{-m_i/n_i}$ is actually constant over the leaves of $\fol$. Meanwhile, the leaves of
$\fol$ are parameterized by a local coordinate $z_i \in \C$ satisfying $z_i^{n_i} = v_i$ and $z_i^{m_i} = u_i$. Thus, we actually
have $z_i = t_i$. In other words, the latter integral becomes
$$
\int_{\Delta} a_{11} (v_i^{1/n_i}, u_i v_i^{-m_i/n_i}) dz_i \wedge d\overline{z}_i \, ,
$$
where $\Delta$ is a disc of (uniform) positive radius around $0 \in \C$. The proposition follows since 
$a_{11} (v_i^{1/n_i}, u_i v_i^{-m_i/n_i})$ is identified with a bounded continuous function.
\end{proof}

\begin{proof}[Proof of Theorem~A]
Statements~(a) and~(b) of Theorem~A were proved in Theorem~\ref{a&b}. In turn assertion~(c) follows from the combination
of Proposition~\ref{Currents_upstairs} and~\ref{Pullingback_Currents}. In fact, Proposition~\ref{Pullingback_Currents}
implies that positive foliated harmonic currents for $\tilf$ on $\widetilde{\mathcal{L} (\fol)}$ are in 1-1 correspondence
with  positive foliated harmonic currents for $\fol$ on $\mathcal{L} (\fol)$ so that the existence and uniqueness of $T$
follows from Proposition~\ref{Currents_upstairs}. This proposition also implies that the self-intersection of $T$ must be zero.

It remains to prove assertion~(d). Besides $T$, integration over any of the invariant lines
$C_1, \ldots,C_k$ also yields foliated currents that are positive and harmonic (indeed closed). Let $T_{C_1}, \ldots ,
T_{C_k}$ denote these integration currents. Since $T$ and $T_{C_1}, \ldots , T_{C_k}$ are clearly independent.
It suffices to check that any positive foliated harmonic current $T'$ is a linear combination of the previous currents.
If $T'$ is as above, up to subtracting a suitable linear combination of $T$ and of $T_{C_1}, \ldots , T_{C_k}$ we can
assume that $T'$ gives mass neither to the invariant lines $C_1, \ldots,C_k$ nor to the Levi-flat $\mathcal{L} (\fol)$.
To complete the proof of Theorem~A, we will show that $T'$ as above is identically zero. For this assume aiming at a contradiction
that $T'$ is not identically zero. Then its support intersects non-trivially $F_1 \setminus (\mathcal{L} (\fol) \cup
C_1 \cup \cdots \cup C_k)$. Since $T'$ is harmonic (equivalently associated with an harmonic measure), there must exist leaves $L$
of $\fol$ contained in the invariant open set $F_1 \setminus (\mathcal{L} (\fol) \cup C_1 \cup \cdots \cup C_k)$ that are
{\it recurrent}, i.e., that accumulate on themselves (see any of \cite{Garnett}, \cite{Etienneghys}, \cite{Candel}). However, this is
impossible since, by construction, the transverse dynamics of $\fol$ on the open set in question is equivalent to the dynamics of the
quasifuchsian (or Fuchian) group $\Gamma$ on its discontinuity set and therefore wandering. This ends the proof of Theorem~A.
\end{proof}


\begin{thebibliography}{Dillo 83}





\bibitem{barth} {\sc W. Barth, K. Hulek, C. Peters \& A. Van de Ven}, {\it Compact complex surfaces, second edition},
Springer-Verlag, Berlin, 2004.


\bibitem{BoNessim} {\sc B. Berndtsson \& N. Sibony}, The $\overline{\partial}$ equation on a positive
current, {\it Invent. Math.}, {\bf 147}, (2002), 371-428.


\bibitem{Birkhoff-1} {\sc G.D. Birkhoff}, A theorem on matrices of analytic functions, {\it Collected Math. Papers},
Amer. Math. Soc, Providence, RI, 1950.


\bibitem{Birkhoff-2} {\sc G.D. Birkhoff}, The generalized Riemann problem for linear differential equations and the
allied problem for linear difference and q-difference equations, {\it Collected Math. Papers}, Amer. Math. Soc,
Providence, RI, 1950.


\bibitem{bowen} {\sc R. Bowen}, Hausdorff dimension of quasicircles, {\it Publ. Math. l'IHES},
{\bf 50}, (1979), 11-25.


%

\bibitem{brunella-L'EnsMath} {\sc M. Brunella}, Courbes enti\`eres et feuilletages holomorphes,
{\it L'Enseignement Math\'ematique}, {\bf 45}, (1999), 195-216.


\bibitem{camacho} {\sc C. Camacho \& P. Sad}, Invariant varieties through singularities of holomorphic vector fields,
{\em Annals of Math.}, {\bf 115}, (1982), 579-595.


\bibitem{Candel} {\sc A. Candel}, The harmonic measures of Lucy Garnett, {\it Adv. Math.}, {\bf 176}, 2, (2003), 187-247.


\bibitem{Demailly} {\sc J.-P. Demailly}, Algebraic criteria for Kobayashi hyperbolic projective varieties and jet
differentials, {\it Proceedings of Symposia in Pure Math.}, Vol. {\bf 62}, 2, (1997), 285-360.





\bibitem{deroinkleptsyn} {\sc B. Deroin \& V. Kleptsyn}, Random conformal
dynamical systems, {\it Geom. Funct. Anal.}, {\bf 17}, 4, (2007), 1043-1105.


\bibitem{sibony_uniqueness} {\sc T.-C. Dinh, V.-A. Nguyen \& N. Sibony},
Unique ergodicity for foliations on compact Kähler surfaces, {\it Duke Math. J.}, {\bf 171}, 13, (2022), 2627-2698.



\bibitem{sibonyetc} {\sc T.-C. Dinh, V.-A. Nguyen \& N. Sibony}, Heat equation and ergodic theorems for
Riemann surface laminations, {\it Math. Ann.}, {\bf 354}, 1, (2012), 331-376.


\bibitem{dinhsibony} {\sc T.-C. Dinh \& N. Sibony}, Unique ergodicity for foliations on $\mathbb{P}^2$ with
an invariant curve, {\it Inventiones mathematicae}, {\bf 211}, (2018), 1-38. 


\bibitem{moreonesibony} {\sc T.-C. Dinh \& N. Sibony}, Some Open Problems on Holomorphic Foliation Theory,
{\it Acta Mathematica Vietnamica}, {\bf 45}, (2020), 103-112.





\bibitem{perez-mol-rosas} {\sc A. Fern\'andez-P\' erez, R. Mol \& R. Rosas}, Chow's theorem for real analytic Levi-flat
hypersurfaces, {\it Bull. Sci. math.}, {\bf 179}, (2022), 1-18.


\bibitem{lebel-perez} {\sc  A. Fern\'andez-P\' erez \& J. Lebl}, {\it Global and local aspects of Levi-flat hypersurfaces},
IMPA Mathematical Publications, 30th Brazilian Mathematics Colloquium, Instituto Nacional de Matemática Pura e Aplicada
(IMPA), Rio de Janeiro, 2015. 

\bibitem{fisher} {\sc W. Fischer \& H. Grauert}, Lokal-triviale Familien kompakter komplexer Mannigfaltigkeiten,
{\it Nachr. Akad. Wiss. G\"{o}ttingen Math.-Phys. Kl. II}, (1965), 89-94.


\bibitem{fornaessSibony} {\sc E. Fornaess \& N. Sibony}, Harmonic currents of finite energy and laminations,
{\it Geom. Funct. Anal.}, {\bf 15}, 5, (2005), 962-1003.


\bibitem{fornaesssibony-2} {\sc E. Fornaess \& N. Sibony}, Unique ergodicity of harmonic currents on singular foliations of $\mathbb{P}^2$, {\it Geom. Funct. Anal.}, {\bf 19}, (2010), 1334-1377.


\bibitem{fornaesssibony-survey} {\sc E. Fornaess \& N. Sibony}, Riemann surface laminations with singularities,
{\it J. Geom. Anal.}, {\bf 18}, (2008), 400-442.


\bibitem{Garnett} {\sc L. Garnett}, Foliations, The ergodic theorem and Brownian motion, {\it J. Funct.
Anal.}, {\bf 51}, (1983), 285-311.


\bibitem{garrandes} {\sc C. P. Garrand\' es}, Laminations by Riemann surfaces in K\"{a}hler surfaces, PhD thesis,
Universidad Complutense de Madrid, Madrid, 2014.


\bibitem{Etienneghys} {\sc E. Ghys}, Topologie des feuilles g\'en\'eriques,
{\it Annals of Math.}, {\bf 141}, (1995), 387-422.



\bibitem{adolfo-IHES} {\sc A. Guillot}, Sur les \'equations d'Halphen et les actions de ${\rm SL}\, (2, \C)$, {\it Publ. Math. IHES} \textbf{105}, 1, (2007), 221-294.

\bibitem{hilbert1} {\sc F. Hirzebruch}, Hilbert modular surfaces, {\it Enseignement Math.},
{\bf 29}, (1973), 183-281.


\bibitem{hilbert2} {\sc F. Hirzebruch \& D. Zagier}, Classification of Hilbert modular surfaces,
{\it Complex Analysis and Geometry, A collection of Papers Dedicated to K. Kodaira},
Cambridge U. Press, (1977), 43-77.



\bibitem{IlyashenkoYakovenko} {\sc Y. Il'yashenko \& S. Yakovenko}, {\it Lectures on analytic
differential equations}, American Mathematical Soc., Rhode Island, 2008.


\bibitem{svetlana} {\sc S. Katok}, {\it Fuchsian groups}, University of Chicago Press, Chicago and London, 1992.


\bibitem{lebel} {\sc J. Lebl}, Algebraic Levi-Flat Hypervarieties in Complex Projective Space,
{\it J. Geom. Anal.}, {\bf 22}, (2012), 410-432.




\bibitem{LinsNeto} {\sc A. Lins-Neto}, Construction of singular holomorphic vector fields and foliations in dimension two, {\it J. Differential Geometry}, {\bf 26}, (1987), 1-31.


\bibitem{Loray} {\sc F. Loray \& J.\, C. Rebelo}, Minimal, rigid foliations by curves on ${\mathbb C}
{\mathbb P} (n)$, {\it Journal of the European Mathematical Society}, {\bf 5}, 2, (2003), 147-201.


\bibitem{Ramis-IHES} {\sc J. Martinet \& J.-P. Ramis}, Probl\`emes de modules pour des \'equations diff\'erentielles
non lin\'eaires du premier ordre, {\it Publ. Math. I.H.E.S.}, {\bf 55}, (1982), 63-164.


\bibitem{kleiniangroups} {\sc K. Matsuzaki \& M. Taniguchi}, {\it Hyperbolic manifolds and Kleinian groups}, Clarendon
Press, Oxford, 1998.


\bibitem{MatteiMoussu} {\sc J.-F. Mattei \& R. Moussu}, Holonomie et int\'egrales premi\`eres, {\em Ann. Scient. Ec. Norm. Sup.},
{\bf 16}, (1983), 469-523.


\bibitem{mcquillan} {\sc M. Mcquillan}, Diophantine approximations and foliations, {\it Publ. IHES}, {\bf 87},
(1998), 121-174.




\bibitem{lgmendes} {\sc L.G. Mendes \& J.V. Pereira}, Hilbert Modular Foliations on the Projective
Plane, {\it Commentarii Mathematici Helvetici}, {\bf 80}, (2005), 243-291.



\bibitem{nag} {\sc S. Nag}, {\it The complex analytic theory of Teichm\"{u}ller spaces}, Canadian Mathematical Society series of monographs and texts, Wiley-Interscience publications, 1988.


\bibitem{nguyen} {\sc V.-A. Nguy\^en}, Singular holomorphic foliations by curves I: integrability of
holonomy cocycle in dimension~$2$, {\it Inventiones mathematicae}, {\bf 212}, (2018), 531-618. 


\bibitem{nguyensurvey} {\sc V.-A. Nguy\^en}, Ergodic theorems for laminations and foliations: recent results and
perspective, {\it Acta Mathematica Vietnamica}, {\bf 46}, (2021), 9-101. 


\bibitem{plante}  {\sc J. Plante},  Foliations with measure preserving holonomy, {\it Ann. of Math.}, {\bf 102}, (1975), 327-361.


\bibitem{RR-Iberoamericana} {\sc J.\,C. Rebelo \& H. Reis}, Uniformizing complex ODEs and applications,
{\it Rev. Mat. Iberoam.}, {\bf 30}, 3, (2014), 799-874.


\bibitem{HelenaandI} {\sc J.\,C. Rebelo \& H. Reis}, Local theory of holomorphic foliations and vector fields, {\it Lecture notes available from} \url{https://arxiv.org/pdf/1101.4309.pdf}.


\bibitem{sullivan} {\sc D. Sullivan}, Cycles for the dynamical study of foliated manifolds and complex
manifolds, {\it Invent. Math.}, {\bf 36}, (1976), 225-255.





\end{thebibliography}
\end{document}